\documentclass[11pt]{amsart}

\usepackage{geometry,graphicx,amssymb,amsmath,amsfonts,bm,tcolorbox,enumitem,amsbsy}
\usepackage{caption,booktabs,array,multirow,cellspace}
\usepackage[all]{xy}
\usepackage{pgfplots}
\usepackage{hyperref}
\usepackage{relsize}
\usepackage{tikz}
\usetikzlibrary{arrows,shapes}
\usetikzlibrary{arrows, decorations.markings,fit}
\usetikzlibrary{calc,3d}
\usepackage{tikz-3dplot}
\usepgfplotslibrary{groupplots}
\usetikzlibrary{matrix, positioning}
\pgfplotsset{compat=1.18}

\tikzset{
	>=stealth',
	punkt/.style={
		rectangle,
		rounded corners,
		draw=black, very thick,
		text width=6.5em,
		minimum height=2em,
		text centered},
	pil/.style={
		->,
		thick,
		shorten <=2pt,
		shorten >=2pt,}
}

\usepackage[normalem]{ulem}
\usepackage{subcaption}
\numberwithin{equation}{section}

\allowdisplaybreaks[3]

\newtheorem{theorem}{Theorem}[section]
\newtheorem{lemma}[theorem]{Lemma}
\newtheorem{assumption}[theorem]{Assumption}
\newtheorem{corollary}[theorem]{Corollary}
\newtheorem{proposition}[theorem]{Proposition}
\theoremstyle{definition}
\newtheorem{definition}[theorem]{Definition}
\newtheorem{remark}[theorem]{Remark}

\newcommand{\red}[1]{{\color{red} #1}}

\newcommand{\R}{\mathbb{R}}

\newcommand{\curl}{{\ensuremath\mathop{\mathrm{curl}\,}}}
\newcommand{\dive}{{\ensuremath\mathop{\mathrm{div}\,}}}

\newcommand{\ip}[1]{\big\langle {#1} \big\rangle}

\newcommand{\bld}[1]{\boldsymbol{#1}}

\newcommand{\bz}{\mathbf{z}}

\newcommand{\bH}{\bld{H}}

\newcommand{\by}{\mathbf{y}}

\newcommand{\bC}{\bld{C}}

\newcommand{\balpha}{{\bm \alpha}}

\newcommand{\bff}{{\mathbf{f}}}
\newcommand{\bgg}{{\mathbf{g}}}

\newcommand{\be}{\bm e}
\newcommand{\beps}{\bm \epsilon}

\newcommand{\bbE}{\mathbb{E}}
\newcommand{\bbR}{\mathbb{R}}
\newcommand{\bbS}{\mathbb{S}}

\newcommand{\calE}{\mathcal{E}}
\newcommand{\calC}{\mathcal{C}}
\newcommand{\calP}{\mathcal{P}}

\newcommand{\calI}{\mathcal{I}}
\newcommand{\calL}{\mathcal{L}}
\newcommand{\calN}{\mathcal{N}}

\newcommand{\dmu}{\mathrm{d}\mu}

\newcommand{\sumone}{\sum_{\ell=1}^{\infty}}

\newcommand{\sumlk}{\sum_{\ell=1}^{\infty}\sum_{k=1}^{2\ell+1}}
\newcommand{\hatdiv}[1]{\widehat{#1}_{\ell,k}^{\dive}}
\newcommand{\hatcurl}[1]{\widehat{#1}_{\ell,k}^{\curl}}

\newcommand{\bHsig}{\bH^{\sigma}(\bbS^2)}
\newcommand{\bHtau}{\bH^{\tau}(\bbS^2)}

\newcommand{\scaleKer}{\psi_{\rho}}
\newcommand{\hatphi}{\widehat{\psi}_\rho}

\definecolor{myred}{RGB}{195,0,0}
\definecolor{myblue}{RGB}{0,90,170}
\definecolor{mygreen}{RGB}{0,140,0}

\definecolor{myreddark}{RGB}{140,0,0}
\definecolor{myredlight}{RGB}{255,100,100}
\definecolor{mybluedark}{RGB}{0,50,120}
\definecolor{mybluelight}{RGB}{100,180,255}
\definecolor{mygreendark}{RGB}{0,90,0}
\definecolor{mygreenlight}{RGB}{100,220,100}

\definecolor{mypurple}{RGB}{120,0,120}
\definecolor{myorange}{RGB}{230,120,0}
\definecolor{mycyan}{RGB}{0,150,150}
\definecolor{myyellow}{RGB}{210,180,0}
\definecolor{mybrown}{RGB}{150,100,50}
\definecolor{mygray}{RGB}{120,120,120}

\title[General degree isoparametric SV]{Vector field multiplier operators and matrix-valued\\ kernel quasi-interpolation}\thanks{The first author was supported in part by National Natural Sicence Foundation of China (No. 12571407), Basic Research Program of Jiangsu (No. BK20252037), and a Jiangsu Shuangchuang Team program (No. JSSCTD202449). The third author is the corresponding author.}

\author[Z. Sun]{Zhengjie Sun}
\address{Nanjing University of Science and Technology, School of Mathematics and Statistics}
\email{sunzhengjie1218@163.com}

\author[B. Huang]{Biao Huang}
\address{School of Mathematics and Statistics, Nanjing University of Science and Technology, China}
\email{biaohuangbest@163.com}

\author[X. Sun]{Xingping Sun}
\address{Department of Mathematics, Missouri State University, Springfield, MO 65897, USA}
\email{xsun@missouristate.edu}

\thanks{All data generated or analyzed during this study are included in this article.}


\begin{document}
	
	\include{aata_VecQI}

	\maketitle
	
	\begin{abstract}
		We develop and analyze a class of matrix-valued spherical-convolution kernels stemming from scaled zonal functions on $\mathbb{S}^2,$ the unit sphere embedded in $\R^3$. The construct of these kernels utilizes the Legendre differential equation and requires less stringent regularity conditions on the original zonal kernels. 
		The induced integral operators are simple Fourier-Legendre multipliers  that not only deliver optimal Sobolev error estimates  (in terms of the scaling parameter) but also yield natural Helmholtz-Hodge decompositions on the $L_2$-tangential vector fields on $\mathbb{S}^2$. Via discretization of the underlying convolution integrals, we harvest a family of vector-valued quasi-interpolants that 
		accomplish our approximation goal in the divergence/curl-free vector field. The quasi-interpolation algorithm is robust against noisy data. 
		The implementation process is adaptive to human-improvision, involving neither evaluating the convolution integrals nor solving systems of linear equations. The computational efficiency and executory robustness of the quasi-interpolation algorithm 
		stand in sharp contrast to the existing kernel-based vector field interpolation method. 

\end{abstract}


\section{Introduction}
\label{sec:intro}

\thispagestyle{empty}
Vector fields on the unit sphere $\bbS^2$ embedded in $\R^3$, play a central role in many areas of science and engineering. Typical examples include the horizontal wind field in atmospheric science, surface ocean currents in oceanography, and magnetic and electric fields in electromagnetics. In many aspects of real world applications, the vector fields are intrinsically defined on a surface and satisfy additional differential constraints arising from underlying physical principles. For instance, incompressible flow models impose a divergence-free condition on the velocity field, while certain electrostatic configurations give rise to curl-free electric fields \cite{Cockburn_2004JCP_locally,Fuselier_2016ComputFluids_high,Guzman_2014IMAJNA_conforming,Guzman_2014MCoM_conforming,Neilan_2021SINUM_divergence,Zhao_2019SINUM_divergence}. On $\bbS^2$, tangential vector fields admit a Helmholtz–Hodge decomposition into uniquely determined divergence-free and curl-free components \cite{Fuselier_2009SINUM_stability,Fuselier_2017IMAJNA_radial}, which provides a powerful diagnostic tool for analyzing phenomena such as cyclones, gyres, and large-scale pressure systems.

Kernel-based methods, especially those stemming from radial basis functions (RBFs) or spherical basis functions (SBFs), lend a versatile toolbox  for scattered data approximation on spheres and other manifolds; see \cite{Drake_2021SISC_partition,Drake_2022SISC_implicit,LeGia_2010SINUM_multiscale,Narcowich_2007JFAA_divergence,Sun_2024SISC_high} and references therein. For vector fields in $\bbR^3$, several authors have developed divergence/curl-free kernels by applying the curl/gradient operator to a scalar kernel, then transpose-multiplying to get a matrix-valued kernel whose columns are analytically divergence-free or curl-free \cite{Lowitzsch_2005AdvCM_matrix,Narcowich_1994MCoM_generalized,Wendland_2009SINUM_divergence}.
Vector spherical harmonics and spherical basis functions have likewise been employed to approximate and decompose tangential vector fields on the sphere \cite{Freeden_1993MMAS_vector,Freeden_2008book_spherical,Ganesh_2011MCoM_pseudospectral,LeGia_2021ToMS_algorithm}. Surface analogues of divergence/curl-free kernels are obtained by tangentializing the curl/gradient operator and then applying them to the restriction to $\bbS^2\times \bbS^2$ 
of a radial kernel on $\bbR^3\times \bbR^3$ \cite{Fuselier_2008MCoM_sobolev,Fuselier_2009MCoM_error,Fuselier_2009SINUM_stability,Narcowich_2007JFAA_divergence}.
These algorithms culminate in finding  vector-valued scattered-data interpolants that accomplish the  vector field approximation task.

To implement such an interpolatory algorithm, one must employ a strictly positive definite function (SPDF) on $\bbS^2;$
see \cite{chen-m-sun, xu-cheney}. In addition, the Fourier-Legendre coefficients of the SPDF employed therein 
must have a sufficiently-high decay rates, which arise from applications of differential operators. These have limited the scope of applicable kernels in the vector field approximation setting. Moreover, the interpolatory algorithm requires inverting large and (often) full interpolation matrices, which poses formidable computational challenges. In the scattered-data interpolation setting, the algorithm is rigid in the sense that one has to overhual the entire process while adding or removing a small number of vector-data. This has ruled of the possibility of human-improvising during the execution stage. Last but not the least, kernel-based interpolation algorithms become less robust  as the number of nodes increases, which becomes necessary  when the preset degree of accuracy is high. This dichotomy is precisely explained in
Schaback's uncertainty relation.\footnote{In a nutshell, Schaback's uncertainty relation asserts that the spectral norm of the inverse of an interpolation matrix is inversely proportional to the approximation order which the underlying interpolant attains.}
Fisher et al. \cite{Gao_2022ACHA_divergence} studied a quasi-interpolation scheme for divergence/curl-free vector field approximation in $\R^d\ (d \ge 2).$ But vector-data must be collected on a full grid in order to implement the quasi-interpolation algorithm, which has hindered its extension to scattered-data settings on compact domains or manifolds. 

In this paper, we establish a new quasi-interpolation framework for scattered-data vector field approximation and Helmholtz–Hodge decomposition on the sphere. 
For a given zonal kernel $\psi_\rho(x\cdot y),\ \rho>0$,  which is the restriction to $\bbS^2 \times \bbS^2$ of a $\rho$-scaled radial kernel on $\bbR^3 \times \bbR^3$,  we define an auxiliary zonal kernel $\kappa$  via an anti-derivative of $\psi_\rho$. By utilizing the Legendre differential equation, we construct divergence-free and curl-free matrix-valued kernels using $\kappa$ and two first-order tensors. In contrast to the existing constructs of  divergence-free and curl-free matrix-valued kernels which features direct applications of surface differential operators to a scalar kernel and transpose multiplying,
our approach requires weaker regularity conditions on the original kernel. Furthermore, our construction process does not require the original kernel to be positive definite and yet it preserves the 
positive-definite property of the original kernel. This has broaden the scope of applicable kernels in this computing environment.
These kernels induce vector-valued spherical convolution operators that act on tangential vector fields as Fourier-Legendre multipliers and preserve the Helmholtz–Hodge decomposition of a target vector field. We employ these convolution operators to approximate tangential vector fields and derive optimal Sobolev error estimates (in terms of the scaling parameter).
We point out here that these convolution operators are merely used for the sake of theoretical analysis. In implementing our quasi-interpolation algorithm, there is no need to evaluate the convolution integrals.  As the final step, we discretize the convolution integrals using QMC designs, which harvests a family of vector-valued quasi-interpolants. We derive Sobolev and mean-square error estimates for the quasi-interpolants. We also show that our quasi-interpolation algorithm is robust against noisy vector-valued data. As an added bonus, our quasi-interpolation algorithm is highly adaptive to various computing environments. We mention specifically here the aspect of human-improvising. A programmer can add or remove any data points cost-free. For scalar fields, related work is carried out in \cite{Sun_2025IMAJNA_spherical, Sun_2025_monte} .

An outline of the paper goes as follows. Section \ref{sec:prelim} devotes to introducing notations and reviewing vector spherical harmonics and basic properties of scalar zonal kernels that serve as building blocks for our divergence-free and curl-free kernels. Section \ref{sec:Construction} presents our new method of constructing matrix-valued kernels whose Fourier-Legendre coefficients decay algebraically and therefore  induce Sobolev spaces of vector fields.  In Section \ref{sec:quasi-interp}, we show that these matrix-valued kernels induce a class of Fourier-Lengedre multiplier operators. Furthermore, by discretizing the underlying integrals based on QMC-designs, we harvest a family of vector-valued quasi-interpolants. Section \ref{sec:err}  devotes to error analysis. We show that the multiplier operators not only deliver optimal Sobolev error estimates  (in terms of the scaling parameter) but also yield natural Helmholtz decompositions on the $L_2$-tangential vector fields on $\mathbb{S}^2$. We also develop the corresponding error analysis for our QMC-based quasi-interpolants, including their robustness against noisy data. Section \ref{sec:numer_exam} presents results of numerical experiments with scaled Gaussian and  Wendland kernels,
which are inline with the theoretical analysis. 

\section{Notations and preliminaries}
\label{sec:prelim}
\subsection{Scalar spherical harmonics}
We work on $\mathbb{S}^2$, the unit sphere embedded in $\R^3$, which we often refer to in the sequel as ``the sphere". Denote $\mu$ the rotational invariant measure on $\mathbb{S}^2$, i.e., the restriction to $\mathbb{S}^2$ of the Lebesgue measure on $\R^3.$
We define
$$L_2=L_2(\mathbb{S}^2):=\{f \; \mbox{is measurable on} \;\mathbb{S}^2: \|f\|_2 < \infty\},$$
in which the norm $\|\cdot\|_2$ is induced by the inner product
\begin{equation*}
	\langle f, g \rangle := \int_{\mathbb{S}^2} f(x)g(x) \, \mathrm{d}\mu(x), \quad f,g \in L_2(\mathbb{S}^2).
\end{equation*}
The standard orthonormal basis for $L_2(\mathbb{S}^2)$ is given by the spherical harmonics. For each integer $\ell \ge 0$ and $k=1, \dots, 2\ell+1$, let $Y_{\ell,k}$ denote the real-valued spherical harmonic of degree $\ell$ and order $k$. Let $\mathcal{H}_\ell := \operatorname{span}\{Y_{\ell,k} : k=1, \dots, 2\ell+1\}$ denote the $(2\ell+1)$-dimensional space of spherical harmonics of degree $\ell$. These form an orthonormal basis of the eigenspace corresponding to the eigenvalue $\lambda_\ell = \ell(\ell + 1)$ of the Laplace--Beltrami operator $\Delta_*$ on the sphere.

The spherical harmonics of degree $\ell$ satisfy the addition formula 
\begin{equation}\label{eq:addition_formula}
	\sum_{k=1}^{2\ell+1} Y_{\ell,k}(x) Y_{\ell,k}(y) = \frac{2\ell+1}{4\pi} P_\ell(x \cdot y),
\end{equation}
where $P_\ell$ is the Legendre polynomial of degree $\ell$ on $[-1,1]$, normalized such that $P_\ell(1)=1$. 

Any $f \in L_2(\mathbb{S}^2)$ has a unique Fourier-Legendre representation of the form
\begin{equation*}
	f(x) = \sum_{\ell=0}^\infty \sum_{k=1}^{2\ell+1} \widehat{f}_{\ell,k} Y_{\ell,k}(x), 
\end{equation*}
with Fourier-Legendre coefficients $\widehat{f}_{\ell,k} = \langle f, Y_{\ell,k} \rangle$. The spherical version of the Plancherel theorem asserts that
\begin{equation*}
	\langle f, g \rangle = \sum_{\ell=0}^\infty \sum_{k=1}^{2\ell+1} \widehat{f}_{\ell,k}\ \overline{\widehat{g}_{\ell,k}}, 
	\quad f,\  g \in L_2(\mathbb{S}^2).
\end{equation*}
For $\sigma \ge 0,$
we define the Sobolev space $H^\sigma(\mathbb{S}^2)$ by
\begin{equation*}
	H^\sigma(\mathbb{S}^2) := \left\{ f \in L_2(\mathbb{S}^2) : \|f\|^2_{H^{\sigma}(\mathbb{S}^2)} < \infty \right\},
\end{equation*}
with norm
\begin{equation}\label{eq:sobolev_norm}
	\|f\|^2_{H^{\sigma}(\mathbb{S}^2)} = \sum_{\ell=0}^\infty \sum_{k=1}^{2\ell+1} (1+\lambda_\ell)^{\sigma} |\widehat{f}_{\ell,k}|^2.
\end{equation}

\subsection{Tangent vector fields} 
In this subsection, we briefly review the $\bld{L}_2$-theory for tangential vector fields on $\mathbb{S}^2$
and introduce notations and terminologies along the way. Our exposition here follows closely those in \cite{Freeden_1993MMAS_vector, Narcowich_2007JFAA_divergence}.
A vector field $\bff: \mathbb{S}^2 \to \mathbb{R}^3$ is called tangential if $x^\top \bff(x)=0$ for all $x \in \mathbb{S}^2$.  Here $\mathbf{v}$ and $\mathbf{v}^\top$  denote respectively a column vector and its transpose.
We consider a point $x \in \mathbb{S}^2$ interchangeably as the vector starting at the origin and ending at $x$, which is exactly the outward unit normal vector at the point $x$. The space of square-integrable tangential vector fields $\bld{L}_2(\mathbb{S}^2)$ is defined by
\begin{equation*}
	\bld{L}_2(\mathbb{S}^2) = \left\{ \bff : \mathbb{S}^2 \to \mathbb{R}^3 \mid 
	\bff \text{ is tangential and } \langle \bff, \bff \rangle < \infty \right\},
\end{equation*}
equipped with the inner product
\begin{equation*}
	\langle \bff, \mathbf{g} \rangle = \int_{\mathbb{S}^2} \bff(x)^{\top} \mathbf{g}(x) \, \mathrm{d}\mu(x), \quad \bff, \mathbf{g} \in \bld{L}_2(\mathbb{S}^2).
\end{equation*}

Let $\nabla_*$  and $\mathbf{L}_*$ denote, respectively, the surface gradient and surface curl operators. It is well known that $\nabla_*^{\top}\nabla_* = \mathbf{L}_*^{\top}\mathbf{L}_* = -\Delta_*$. 
For a scalar field $f$, $\nabla_* f$ is tangential. Furthermore, the surface curl/gradient of a scalar function is divergence/curl-free. 

Vector spherical harmonics provide a natural orthonormal basis for tangent vector fields, which are constructed from scalar spherical harmonics by applying the surface gradient and surface curl operators. For each $\ell \ge 1$ and $k=1, \dots, 2\ell+1$, the following two families of vector spherical harmonics are defined by:
\begin{equation}\label{eq:def_ylk}
	\mathbf{y}_{\ell,k} = \frac{\mathbf{L}_* Y_{\ell,k}}{\sqrt{\ell(\ell+1)}}, \quad
	\mathbf{z}_{\ell,k} = \frac{\nabla_* Y_{\ell,k}}{\sqrt{\ell(\ell+1)}}.
\end{equation}
Respectively, they are orthonormal bases for the divergence-free and curl-free tangential vector fields on $\mathbb{S}^2$. Together, they form a complete orthonormal basis for $\bld{L}_2(\mathbb{S}^2)$, 
under which  each $\bff \in \bld{L}_2(\mathbb{S}^2)$ admits the expansion
\begin{equation*}
	\bff(x) = \sum_{\ell=1}^\infty \sum_{k=1}^{2\ell+1} \left( \widehat{\bff}^{\mathrm{div}}_{\ell,k} \mathbf{y}_{\ell,k}(x) + \widehat{\bff}^{\mathrm{curl}}_{\ell,k} \mathbf{z}_{\ell,k}(x) \right),
\end{equation*}
where the divergence-free and curl-free Fourier-Legendre coefficients are given by
\begin{equation*}
	\widehat{\bff}^{\mathrm{div}}_{\ell,k} := \langle \bff, \mathbf{y}_{\ell,k} \rangle, \quad 
	\widehat{\bff}^{\mathrm{curl}}_{\ell,k} := \langle \bff, \mathbf{z}_{\ell,k} \rangle.
\end{equation*}
For $\sigma \in \mathbb{N}_0$, we define the vectorial Sobolev space $\bld{H}^{\sigma}(\mathbb{S}^2)$ as the set of all $\bff \in \bld{L}_2(\mathbb{S}^2)$ such that
\begin{equation*}
	\|\bff\|_{\bld{H}^{\sigma}(\mathbb{S}^2)}^2 = \sum_{\ell=1}^\infty \sum_{k=1}^{2\ell+1} (1+\lambda_\ell)^{\sigma} \left( |\widehat{\bff}^{\mathrm{div}}_{\ell,k}|^2 + |\widehat{\bff}^{\mathrm{curl}}_{\ell,k}|^2 \right) < \infty.
\end{equation*}
We denote by $\bld{H}^{\sigma}_{\mathrm{div}}(\mathbb{S}^2)$ and $\bld{H}^{\sigma}_{\mathrm{curl}}(\mathbb{S}^2)$ the divergence-free and curl-free subspaces of $\bld{H}^{\sigma}(\mathbb{S}^2)$, respectively.

\subsection{Zonal kernels and native spaces}
We employ as our main approximation tools a family of zonal kernels. A kernel $\Phi : \mathbb{S}^2 \times \mathbb{S}^2 \to \mathbb{R}$ is called \emph{zonal} if it is invariant under orthogonal transformations. Equivalently, there exists a univariate function $\psi : [-1,1] \to \mathbb{R}$, such that
\[
\Phi(x,y) = \psi(x \cdot y), \quad (x,y) \in \bbS^2\times \bbS^2.
\]
A zonal kernel $\Phi$ enjoys a simple 
Fourier--Legendre expansion:
\begin{equation}\label{eq:Four-LegendreSeries}
	\psi(x \cdot y) = \sum_{\ell=0}^\infty \frac{2\ell+1}{4\pi} \widehat{\psi}(\ell) P_\ell(x \cdot y),
\end{equation}
where the Fourier--Legendre coefficients are given by
\begin{equation}\label{eq:def_hatpsi}
	\widehat{\psi}(\ell) = \frac{2\pi}{2\ell+1} \int_{-1}^1 \psi(t) P_\ell(t) \, \mathrm{d}t.
\end{equation}
Applying the addition formula \eqref{eq:addition_formula}, the kernel may be rewritten in terms of spherical harmonics as
\begin{equation}\label{eq:zonal_kernel_SH}
	\psi(x \cdot y) = \sum_{\ell=0}^\infty \widehat{\psi}(\ell) \sum_{k=1}^{2\ell+1} Y_{\ell,k}(x) Y_{\ell,k}(y).
\end{equation}
This identity connects  a zonal kernel $\Phi$ to its univariate-function representor $\psi$ seamlessly. In what follows, we make no effort trying to distinguish one from the other when referring to a zonal kernel.
If $\widehat{\psi}(\ell) > 0$ for all $\ell \ge 0$, then $\psi$ is strictly positive definite and induces the following reproducing kernel Hilbert space (RKHS) $\mathcal{N}_{\psi}$: 
\begin{equation*}
	\mathcal{N}_{\psi} = \left\{ f \in \mathcal{D}'(\mathbb{S}^2) : \sum_{\ell=0}^\infty \sum_{k=1}^{2\ell+1} \frac{|\widehat{f}_{\ell,k}|^2}{\widehat{\psi}(\ell)} < \infty \right\},
\end{equation*}
where $\mathcal{D}'(\mathbb{S}^2)$ denotes the space of 
Schwarz-class distributions on the sphere. In the literature, we often call $\mathcal{N}_{\psi}$
the \emph{native space} of $\psi$. 
It is evident that if the coefficients decay algebraically as $\widehat{\psi}(\ell) \sim (1+\lambda_\ell)^{-\sigma}$, then $\mathcal{N}_{\psi}$ is norm-equivalent to the Sobolev space $H^\sigma(\mathbb{S}^2)$.

To facilitate multiscale analysis, we consider families of scaled kernels. Let $\phi \in C[0,\infty)$. For a scaling parameter $\rho \in (0,1)$, we define on the sphere the following kernel with the scaling parameter $\rho:$
\begin{equation}\label{eq:scaled_zonal_kernel}
	\psi_\rho(x \cdot y) := C_{\psi,\rho} \phi\left(\rho^{-1}\|x-y\|\right), \quad (x,y) \in \bbS^2\times \bbS^2,
\end{equation}
where $C_{\psi,\rho}$ is a normalizing constant such that $\int_{\mathbb{S}^2} \psi_\rho(x \cdot y) \, \mathrm{d}\mu(x) = 1$. The scaled kernel admits the expansion
\begin{equation*}
	\psi_\rho(x \cdot y) = \sum_{\ell=0}^\infty \widehat{\psi}_\rho(\ell) \sum_{k=1}^{2\ell+1} Y_{\ell,k}(x) Y_{\ell,k}(y).
\end{equation*}
The associated native space $\mathcal{N}_{\psi_\rho}$ is equipped with the $\rho$-dependent norm
\begin{equation}\label{eq:NativeSpace_scaled}
	\|f\|_{\psi_\rho}^2 = \sum_{\ell=0}^\infty \sum_{k=1}^{2\ell+1} \frac{|\widehat{f}_{\ell,k}|^2}{\widehat{\psi}_\rho(\ell)}.
\end{equation}

To facilitate error analysis, we make the following assumptions on the Fourier-Legendre coefficients of the underlying scaled zonal kernels.

\begin{assumption}\label{assump:kernel}
	There exist  $0 <  \rho_0 < 1$, $m \ge 0,$ and $c_1, c_2 > 0$, such that
	\begin{subequations}\label{eq:assump_eq1}
		\begin{align}
			&|1 - \widehat{\psi}_\rho(\ell)| \le c_1 (\ell \rho)^m, \quad 0 < \rho <  \rho_0, \quad  0 \le \ell \le \lfloor \rho^{-1}-1 \rfloor:=\ell_\rho.  \label{Assumption-1}\\
			&\widehat{\psi}_\rho(\ell) \le c_2, \quad \ell > \ell_\rho.      \label{Assumption-2}
		\end{align}    
	\end{subequations}
\end{assumption}

\begin{assumption}\label{assump2}
	For the scaled zonal kernel $\scaleKer$, there exist constants $0<c_3\leq c_4<\infty$ such that
	\begin{equation}\label{eq:kernelFourierCond}
		c_3(1+\rho\ell)^{-2s}\leq \widehat{\psi}_\rho(\ell)\leq c_4(1+\rho\ell)^{-2s},~~\ell>0.
	\end{equation}
\end{assumption}

Assumption \ref{assump:kernel} is first instituted in 
\cite{Sun_2025IMAJNA_spherical}. The most commonly used scaled zonal kernels, such as
spherical Poisson kernels, Gaussian and compactly supported Wendland kernels, satisfy this assumption.
Using Assumption \ref{assump2},   Le Gia et al. \cite[Lemma~3.1]{LeGia_2010SINUM_multiscale} proved that the two norms $\| \cdot \|_{\psi_\rho}$ and $\| \cdot \|_{H^s(\mathbb{S}^2)}$
are  equivalent.

\begin{lemma}\label{lem:Equiv_Sobolev_Native}
	Let $\rho \in (0,1]$. For any $f \in H^s(\mathbb{S}^2)$, the following norm equivalences hold:
	\begin{equation*}
		\|f\|_{\psi_\rho} \le c_3^{-1/2} \|f\|_{H^s(\mathbb{S}^2)} \quad \text{and} \quad
		\|f\|_{\psi_1} \le (c_4/c_3)^{1/2} \rho^{-s} \|f\|_{\psi_\rho}.
	\end{equation*}
\end{lemma}

At the end of this subsection, we review the notion of positive definiteness for matrix-valued kernels, which is essential for vector field interpolation but only plays a peripheral role in quasi-interpolation studied herein. A kernel $\mathbf{\Psi}: \mathbb{S}^2 \times \mathbb{S}^2 \to \mathbb{R}^{3 \times 3}$ is \emph{positive semidefinite} if, for any finite set of points $X = \{x_j\}_{j=1}^N \subset \mathbb{S}^2$ and any vectors $\boldsymbol{\alpha}_j \in T_{x_j}\mathbb{S}^2$,
\begin{equation*}
	\sum_{j,k=1}^N \boldsymbol{\alpha}_j^\top \mathbf{\Psi}(x_j, x_k) \boldsymbol{\alpha}_k \ge 0.
\end{equation*}
If the quadratic form vanishes only when $\boldsymbol{\alpha}_j = \mathbf{0}$ for all $j$, $\mathbf{\Psi}$ is termed \emph{positive definite}.

\subsection{The addition formula for vector spherical harmonics} Serving as the pillar of
construction of divergence-free and curl-free kernels is the following addition formula for vector spherical harmonics.

\begin{lemma}\label{lem:vec_addition}
	Let $\{\mathbf{y}_{\ell,k}\}$ and $\{\mathbf{z}_{\ell,k}\}$ be, respectively, the orthonormal bases of divergence-free and curl-free vector spherical harmonics of degree $\ell \ge 1$,  Define the tensor kernels
	\begin{equation*}
		\mathcal{S}_\ell(x,y) := \sum_{k=1}^{2\ell+1} \mathbf{y}_{\ell,k}(x) \mathbf{y}_{\ell,k}(y)^{\top}, \quad
		\mathcal{T}_\ell(x,y) := \sum_{k=1}^{2\ell+1} \mathbf{z}_{\ell,k}(x) \mathbf{z}_{\ell,k}(y)^{\top}.
	\end{equation*}
	Then we have
	\begin{align*}
		\mathcal{S}_\ell(x,y) &= \frac{2\ell+1}{4\pi\ell(\ell+1)} \left[ P_\ell''(t) \mathbf{Q}(x,y) + P_\ell'(t) \mathbf{R}(x,y) \right], \\
		\mathcal{T}_\ell(x,y) &= \frac{2\ell+1}{4\pi\ell(\ell+1)} \left[ P_\ell''(t) \mathbf{V}(x,y) + P_\ell'(t) \mathbf{W}(x,y) \right],
	\end{align*}
	in which $t := x \cdot y$, and $\mathbf{Q, R, V, W}$ are matrix-valued functions defined by
	\begin{align*}
		\mathbf{Q}(x,y) &:= -(x \times y)(x \times y)^{\top}, & 
		\mathbf{R}(x,y) &:= t\mathbf{I} - y x^{\top}, \\
		\mathbf{V}(x,y) &:= (y - t x)(x - t y)^\top, &
		\mathbf{W}(x,y) &:= (\mathbf{I} - x x^\top)(\mathbf{I} - y y^\top).
	\end{align*}
\end{lemma}

\begin{proof}
	To derive the first formula, we write first
	\begin{equation*}
		\mathcal{S}_\ell(x,y) = \frac{1}{\ell(\ell+1)} \mathbf{L}_{*,x} \mathbf{L}_{*,y}^{\top} \left( \sum_{k=1}^{2\ell+1} Y_{\ell,k}(x) Y_{\ell,k}(y) \right).
	\end{equation*}
	Applying the addition formula for scalar spherical harmonics within the parentheses above, 
	we get
	\begin{equation}\label{eq:Sl_operator_form}
		\mathcal{S}_\ell(x,y) = \frac{2\ell+1}{4\pi\ell(\ell+1)} \mathbf{L}_{*,x} \mathbf{L}_{*,y}^{\top} P_\ell(x \cdot y).
	\end{equation}
	Recall $\mathbf{L}_{*,y} = y \times \nabla_y$, which leads to
	\begin{equation*}
		\mathbf{L}_{*,y}^{\top} P_\ell(x \cdot y) = (y \times \nabla_y)^\top P_\ell(t) = P_\ell'(t) (y \times x)^\top.
	\end{equation*}
	Substituting this into \eqref{eq:Sl_operator_form} and simplifying, we obtain
	\begin{equation*}
		\mathbf{L}_{*,x} \left[ P_\ell'(t) (y \times x)^\top \right] = (\mathbf{L}_{*,x} P_\ell'(t)) (y \times x)^\top + P_\ell'(t) \mathbf{L}_{*,x} (y \times x)^\top.
	\end{equation*}
	The desired formula for $\mathcal{S}_\ell$ then follows  from the identities:
	\[
	\mathbf{L}_{*,x} P_\ell'(t) = P_\ell''(t) (x \times y), \quad \mathbf{L}_{*,x} (y \times x)^\top = t\mathbf{I} - y x^\top.
	\]
	The derivation of the second formula  proceeds analogously.
	Writing
	\begin{equation*}
		\mathcal{T}_\ell(x,y) = \frac{2\ell+1}{4\pi\ell(\ell+1)} \nabla_{*,x} \nabla_{*,y}^{\top} P_\ell(x \cdot y),
	\end{equation*}
	and recalling $\nabla_{*,y} = (\mathbf{I} - yy^\top)\nabla_y$, we obtain $\nabla_{*,y} P_\ell(t) = P_\ell'(t)(x - ty)$. We then utilize the identities
	\begin{equation*}
		\nabla_{*,x} P_\ell'(t) = P_\ell''(t)(y - tx), \quad \nabla_{*,x} (x - ty)^\top = (\mathbf{I} - xx^\top)(\mathbf{I} - yy^\top),
	\end{equation*}
	to complete the proof.
\end{proof}



\section{New Construction Method}
\label{sec:Construction}
To construct a divergence/curl-free kernel
from a scalar zonal kernel $\psi$, the existing method entails direct applications of 
surface differential operators \cite{Fuselier_2009SINUM_stability,Narcowich_2007JFAA_divergence}:
\begin{equation}\label{eq:standard_construction}
	\Psi_{\dive} := \mathbf{L}_{*,x}\mathbf{L}_{*,y}^{\top}[\psi(x\cdot y)] \quad \text{and} \quad \Psi^{\curl} := \nabla_{*,x}\nabla_{*,y}^{\top}[\psi(x\cdot y)].
\end{equation}
The resulted kernels admit the expansions
\begin{align*}
	\Psi_{\dive}(x,y) &= \sum_{\ell=1}^\infty \ell(\ell+1)\widehat{\psi}(\ell) \sum_{k=1}^{2\ell+1}\mathbf{y}_{\ell,k}(x) \mathbf{y}_{\ell,k}^{\top}(y), \\
	\Psi_{\curl}(x,y) &= \sum_{\ell=1}^\infty \ell(\ell+1)\widehat{\psi}(\ell) \sum_{k=1}^{2\ell+1}\mathbf{z}_{\ell,k}(x) \mathbf{z}_{\ell,k}^{\top}(y).
\end{align*}
The multiplier $\ell(\ell+1)$ puts a significant burden on the Fourier-Legendre coefficients $\widehat{\psi}(\ell)$,
which has ruled out some otherwise useful kernels from consideration. In fact, the spectral norms of the matrices
\[
\sum_{k=1}^{2\ell+1}\mathbf{y}_{\ell,k}(x) \mathbf{y}_{\ell,k}^{\top}(y), \quad \sum_{k=1}^{2\ell+1}\mathbf{z}_{\ell,k}(x) \mathbf{z}_{\ell,k}^{\top}(y),
\]
are asymptotically in the order of $\ell$ in view of Lemma 3.5 in \cite{Fuselier_2009MCoM_error} . In order for the two infinite series above to converge uniformly on $\mathbb{S}^2 \times \mathbb{S}^2$, it is typical that 
$\sum_{\ell=1}^\infty \ell^3\  \widehat{\psi}(\ell) < \infty.$
To relax the constraint, we propose and study a new construction method utilizing Legendre differential equation and an integral primitive of the scalar kernel. Let $\psi$ be a zonal kernel 
with $\widehat{\psi}(0)=1$. We define the auxiliary function $\kappa: [-1, 1] \to \mathbb{R}$ by
\begin{equation}\label{eq:def_kappa}
	\kappa(t) = \frac{1}{4\pi(1-t^2)}\int_{-1}^t \left(1-4\pi\psi(s)\right)\mathrm{d} s.
\end{equation}
The function $\kappa$ is continuous on $[-1,1]$. Set $\beta(t):=\int_{-1}^t \left(1-4\pi\psi(s)\right)\mathrm{d} s.$ Then $\beta(-1)=0.$
The normalization condition that $2\pi\int_{-1}^1\psi(s)\mathrm{d} s = 1$ ensures that
$\beta(1)=0.$
Thus, limits of the function $\kappa$ at both end points are of the type $\frac{0}{0}$. A simple application of L'H\^opital's rule shows that these limits exist.

With this auxiliary function $\kappa$ at hand, we define the new divergence/curl-free kernels respectively by
\begin{subequations}\label{eq:proposed_kernels}
	\begin{align}
		\Psi^\dive(x,y) &= \kappa'(t)\mathbf{Q} + \kappa(t)\mathbf{R}, \label{eq:divfree_kernel} \\
		\Psi^\curl(x,y) &= \kappa'(t)\mathbf{V} + \kappa(t)\mathbf{W}, \label{eq:curlfree_kernel}
	\end{align}
\end{subequations}
in which the scalar $t$ and the matrices $\mathbf{Q}, \mathbf{R}, \mathbf{V}, \mathbf{W}$ are defined as in Lemma \ref{lem:vec_addition}. These matrix-valued kernels
are symmetric and have simple
Fourier-Legendre expansions as shown in Theorem \ref{thm:Four_expan_divcurl} preceded by Lemma \ref{lem:Sxy}. 

\begin{lemma}\label{lem:Sxy}
	Suppose $\psi \in C[-1,1]$ such that $\sum_{\ell=1}^\infty \ell\,|\widehat{\psi}(\ell)| < \infty$. Let $\mathcal{S}_\ell(x,y)$ and $\mathcal{T}_\ell(x,y)$ be defined as in Lemma~\ref{lem:vec_addition}. Then the two series of matrix-valued functions
	\[
	S(x,y) := \sum_{\ell=1}^\infty \widehat{\psi}(\ell)\mathcal{S}_\ell(x,y), \quad T(x,y) := \sum_{\ell=1}^\infty \widehat{\psi}(\ell)
	\mathcal{T}_\ell(x,y),\]
	converge absolutely and uniformly on $\mathbb{Y}:=\left(\mathbb{S}^2\times \mathbb{S}^2\right) \setminus \{(x,y): x = \pm y\}.$
\end{lemma}

\begin{proof}
	We will prove the first result. The proof of the second is parallel. Define
	\begin{equation}\label{eq:defalp_zeta}
		\alpha(t) := \frac{1}{4\pi}\sum_{\ell=1}^\infty \frac{2\ell+1}{\ell(\ell+1)}\ \widehat{\psi}(\ell) P_\ell'(t),
		\quad
		\zeta(t) := \frac{1}{4\pi}\sum_{\ell=1}^\infty (2\ell+1)\ \widehat{\psi}(\ell) P_\ell(t), \quad t\in[-1,1].
	\end{equation}
	Using the obvious inequalities ${\displaystyle \max_{-1 \le t \le 1} |P_\ell(t)| \le 1}$ and ${\displaystyle \max_{-1 \le t \le 1} |P_\ell'(t)| \le \ell(\ell+1)/2}$, we have
	\[
	(2\ell+1)\ \left|\widehat{\psi}(\ell)\right| \max_{-1 \le t \le 1}\left|\ P_\ell(t)\right|
	\le C\,\ell\,|\widehat{\psi}(\ell)|,
	\quad
	\frac{2\ell+1}{\ell(\ell+1)}\left|\widehat{\psi}(\ell)\right|\  \max_{-1 \le t \le 1}\left|P_\ell'(t)\right|
	\le C\,\ell\,|\widehat{\psi}(\ell)|.
	\]
	Here $C>0$ is a constant independent of $\ell$ and $t$. Since $\sum_{\ell\ge 1}\ell\,|\widehat{\psi}(\ell)|<\infty$, both series in \eqref{eq:defalp_zeta} converge absolutely and uniformly on $[-1,1]$. Therefore, $\alpha$ and $\zeta$ are continuous on $[-1,1]$.
	
	Recall that the Legendre polynomial $P_\ell$
	satisfies the following differential equation:
	\begin{equation}\label{legendre}
		P_\ell''(t) =(1-t^2)^{-1} \left[2tP_\ell'(t) - \ell(\ell+1)P_\ell(t)\right], \quad t \in (-1,1).
	\end{equation}
	For $(x,y) \in\mathbb{Y}$ and therefore $t = x\cdot y \in (-1,1)$, we use  
	Lemma~\ref{lem:vec_addition} and Equation \eqref{legendre} to write
	\begin{align}\label{alpha-zeta}
		S(x,y) = &\frac{1}{4\pi}\sum_{\ell=1}^\infty\frac{2\ell+1}{\ell(\ell+1)} \widehat{\psi}(\ell)\,
		\Big(P_\ell''(t)\mathbf{Q} + P_\ell'(t)\mathbf{R}\Big)\\ \nonumber
		= &\frac{2t\alpha(t) - \zeta(t)}{1-t^2}\mathbf{Q} + \alpha(t)\mathbf{R}.
	\end{align}
	This proves the desired result.
\end{proof}

\begin{theorem}\label{thm:Four_expan_divcurl}
	Under the assumptions of Lemma \ref{lem:Sxy},
	we have
	\begin{align}
		&\Psi^{\dive}(x,y) = \sum_{\ell=1}^\infty \widehat{\psi}(\ell)\sum_{k=1}^{2\ell+1}\mathbf{y}_{\ell,k}(x) \mathbf{y}_{\ell,k}(y)^{\top},\label{eq:Four_divekernel}
		\\
		&\Psi^{\curl}(x,y) = \sum_{\ell=1}^\infty \widehat{\psi}(\ell)\sum_{k=1}^{2\ell+1}\mathbf{z}_{\ell,k}(x) \mathbf{z}_{\ell,k}(y)^{\top},\label{eq:Four_curlkernel}
	\end{align}
	in which the two series of matrix-valued functions converge absolutely and uniformly on $\mathbb{S}^2 \times \mathbb{S}^2$.
\end{theorem}

\begin{proof}
	We will prove \eqref{eq:Four_divekernel}; the proof of \eqref{eq:Four_curlkernel} is analogous. Throughout the proof, we refer to and make use of the pertinent equations and notations in the statement and proof of Lemma~\ref{lem:Sxy}.  We will break the proof in two steps. 
	\medskip
	
	\noindent Step I, we prove that  $\Psi^{\dive}(x,y) = S(x,y),\, \forall\  (x,y) \in \mathbb{Y}$.
	By \eqref{eq:divfree_kernel},  \eqref{eq:def_kappa}, and \eqref{alpha-zeta}, it suffices to show that 
	\begin{equation}\label{kappa-derivative}
		\kappa(t)=\alpha(t),\quad \kappa'(t)=\frac{2t\alpha(t) - \zeta(t)}{1-t^2}, \quad t \in (-1,1).  
	\end{equation}
	Starting with the Legendre equation,
	\[
	\frac{d}{dt}\big((1-t^2)P_\ell'(t)\big) = -\ell(\ell+1)P_\ell(t),
	\]
	we multiply both sides of the equation by $\frac{2\ell+1}{4\pi\,\ell(\ell+1)}\widehat{\psi}(\ell)$ and sum over $\ell\ge 1$. We then integrate from $-1$ to $t$ the resulted series term-by-term to obtain
	\begin{equation} \label{1-x2}
		(1-t^2)\alpha(t)
		= \frac{(1-t^2)}{4\pi}\sum_{\ell=1}^\infty \frac{2\ell+1}{\ell(\ell+1)}\  \widehat{\psi}(\ell) P_\ell'(t)
		= -\int_{-1}^t \zeta(s)\,ds. 
	\end{equation}
	Use the identity $\zeta(s)=\psi(s)-\frac{1}{4\pi}$
	to rewrite \eqref{1-x2} as
	\[
	(1-t^2)\alpha(t)
	= -\int_{-1}^t\Big(\psi(s)-\frac{1}{4\pi}\Big)\,ds
	= \frac{1}{4\pi}\int_{-1}^{t}\big(1-4\pi\psi(s)\big)\,ds.
	\]
	This shows that $\alpha(t)=\kappa(t)$ for all $t\in(-1,1)$, which is the first equation in \eqref{kappa-derivative}.
	Differentiating \eqref{1-x2} side-by-side  yields $(1-t^2)\kappa'(t) - 2t\kappa(t) = -\zeta(t)$, 
	which, for $t\in(-1,1)$,  is equivalent to the second equation of \eqref{kappa-derivative}. 
	\medskip

	\noindent Step II, we show that $\Psi^{\dive}(x,y) $ extends to be a continuous matrix-valued function on $\mathbb{S}^2 \times \mathbb{S}^2$. It suffices to show that both terms  $\kappa'(t)\mathbf{Q}$  and $\kappa(t)\mathbf{R}$ on the right hand side of  \eqref{eq:divfree_kernel} extends continuously to $\mathbb{S}^2 \times \mathbb{S}^2$.  By the discussion following \eqref{eq:def_kappa}, the function $\kappa$ extends continuously to $[-1,1]$. Thus,  so does $\kappa \mathbf{R}$ to $\mathbb{S}^2 \times \mathbb{S}^2$.

	It remains to show that $\kappa' \mathbf{Q}$ extends continuously to $\mathbb{S}^2 \times \mathbb{S}^2$. 
	The identity
	\[
	(1-t^2)\kappa'(t)=2t\,\kappa(t)-\zeta(t)
	\]
	implies that $(1-t^2)\kappa'(t)$ extends continuously to $[-1,1]$. Moreover, we have
	\begin{equation} \label{important-limit}
		\lim_{t \rightarrow \pm 1} (1-t^2)\kappa'(t)=\lim_{t \rightarrow \pm 1}\left[2t \kappa(t)-\zeta(t)\right] =\pm 2 \kappa(\pm 1)-\zeta(\pm 1)  =0. 
	\end{equation}
	Since $\mathbf{Q}(x,y)$ is a rank-1 matrix, we have
	\[
	\|\mathbf{Q}(x,y)\|_2 = \|x \times y\|_2^2 = \sin^2 \theta =1 - t^2, 
	\]
	where $\theta$ is the smaller angle between the two vectors $x$ and $y$. It then follows from  \eqref{important-limit} that
	\[
	\lim_{(x,y) \rightarrow (x, \pm x)} \|\kappa'(t)\mathbf{Q}(x,y)\|_2   =\lim_{t \rightarrow \pm 1}(1-t^2)|\kappa'(t)|  =0. 
	\]
	(Here readers' discretion is advised that  
	the same notation $\| \cdot \|_2$ has been used to indicate the Euclidean norm of a vector as well as the spectral norm of a matrix.)
	This shows that $\Psi^{\dive}(x,y) $ extends continuously to $\mathbb{S}^2 \times \mathbb{S}^2$. The absolute and uniform convergence follows from a routine application of Tauberian theorems. 
\end{proof}

\begin{remark}
	For numerical stability, it is advantageous to absorb the term $(1-t^2)^{-1}$ into the matrix $\mathbf{Q}$. Defining the scaled matrix $\widetilde{\mathbf{Q}} := (1-t^2)^{-1} \mathbf{Q}$ 
	allows the divergence-free kernel to be evaluated as
	\begin{equation*}
		\Psi^{\dive}(x,y) = \Big(\frac{1}{4\pi}-\psi(t)+2t\kappa(t)\Big)\widetilde{\mathbf{Q}} + \kappa(t)\mathbf{R},
	\end{equation*}
	which avoids explicit cancellation errors near the diagonal $x=y$.
\end{remark}

The new construction method generalizes and improves upon the existing vector spherical harmonic methods for divergence-free and curl-free approximation discussed in \cite{LeGia_2021ToMS_algorithm,Wahba_1982Inbook_vector}. 
Furthermore, the incorporation of  scaled zonal kernels offers flexibility for multiscale settings and provides additional robustness of the quasi-interpolation algorithm.

Let $N \in \mathbb{N}$. Let $\balpha_j\ (1 \le j \le N)$ be tangential vectors. For any $N$ distinct points $x_j \in \mathbb{S}^2,$ we have
by Theorem \ref{thm:Four_expan_divcurl} that
\begin{equation}
	\begin{aligned}
		\sum^N_{j,k}\balpha_k^\top \Psi^{\dive}(x_k,x_j)\balpha_j &= \sum_{j,k}\sumone\widehat{\psi}(\ell)\sum_{k=1}^{2\ell+1}\balpha_k^\top \by_{\ell,k}(x_k)\by^\top_{\ell,k}(x_j)\balpha_j\\
		&= \sumone\widehat{\psi}(\ell)\sum_{k=1}^{2\ell+1}\left|\sum_{j=1}^N\balpha_j^\top\by_{\ell,k}(x_j)\right|^2.
	\end{aligned}
\end{equation}
Hence, the matrix-valued kernel $\Psi^{\dive}(x,y)$ is (strictly) positive definite whenever the scalar kernel $\psi(x,y)$ is.


The native space associated with $\Psi^{\dive}$ is defined as
\begin{equation}\label{eq:NativeSpace}
	\calN_{\Psi^{\dive}} := \left\{\bff \in \bH_{\dive}^0(\bbS^2) : \sumlk \frac{|\hatdiv{\bff}|^2}{\widehat{\psi}(\ell)} < \infty \right\}.
\end{equation}
Under the decay condition $\widehat{\psi}(\ell) \sim (1+\lambda_\ell)^{-\sigma}$,  the two spaces  $\calN_{\Psi^{\dive}}, \bH_{\dive}^\sigma$ are norm equivalent. 
The native space for $\Psi^{\curl}$ is defined analogously. The native space  of the combined kernel $\Psi$ is defined to be the closure of $\bld{L}_2$ vector-valued functions under norm induced by the inner product
\begin{align*}
	\ip{\bff,\bgg}_{\calN_\Psi} = \sumlk \left( \frac{\hatdiv{\bff}\hatdiv{\bgg}}{\widehat{\psi}(\ell)} + \frac{\hatcurl{\bff}\hatcurl{\bgg}}{\widehat{\psi}(\ell)} \right).
\end{align*}

\subsection{Kernels}
We focus on kernels derived from the restriction of radial basis functions on $\mathbb{R}^3$ to the sphere $\mathbb{S}^2$, specifically the Poisson, Gaussian, and compactly supported Wendland families. Previous analyses \cite{Sun_2025IMAJNA_spherical,Sun_2025_monte} establish that the Poisson kernel satisfies Assumption \ref{assump:kernel} with $m=1$, while the Gaussian and Wendland kernels satisfy the assumption with $m=2$. Kernels of higher regularity may be constructed via linear combinations of these base functions (see, e.g., \cite[Prop.~3.10]{Sun_2025IMAJNA_spherical} and \cite[Sec.~4]{Franz_2023IMAJNA_multilevel}).

To expedite numerical implementation of the divergence/curl-free kernels defined in \eqref{eq:divfree_kernel} and \eqref{eq:curlfree_kernel}, 
we provide closed form solutions of the auxiliary functions $\kappa(t)$ for the most commonly employed radial kernels 
in Table \ref{tab:SBFs}.

\begin{table}[htbp]
	\centering
	\caption{Definitions of the scalar kernels $\psi_\rho$ and auxiliary functions $\kappa$. For the Wendland functions, the radial distance is defined as $r = \frac{1}{\rho}\sqrt{2-2t}$.}
	\label{tab:SBFs}
	\renewcommand{\arraystretch}{2.0}
	\setlength{\tabcolsep}{6pt}
	\small 
	\begin{tabular}{l l l}
		\toprule
		\textbf{Type} & \multicolumn{1}{c}{$\psi_\rho(t)$} & \multicolumn{1}{c}{$\kappa(t)$} \\
		\midrule
		Poisson 
		& $\displaystyle \frac{1-\alpha^2}{4\pi(1-2\alpha t+\alpha^2)^{3/2}}$ 
		& $\displaystyle \frac{1+\alpha t - \frac{1-\alpha^2}{\sqrt{1-2\alpha t+\alpha^2}}}{4\pi\alpha(1-t^2)},~\alpha = 1-\rho$ \\\addlinespace\hline
		\addlinespace
		Gaussian
		& $\displaystyle \frac{1}{2\pi\rho^2} e^{-\frac{1-t}{\rho^2}}$ 
		& $\displaystyle \frac{(1+t)(1-e^{-\frac{2}{\rho^2}}) -2e^{-\frac{1-t}{\rho^2}}+2e^{-\frac{2}{\rho^2}}}{4\pi (1-t^2)}$ \\\addlinespace\hline
		\addlinespace
		$\text{WE}_{3,1}$
		& $\displaystyle \frac{7(1-r)_+^{4}(4r+1)}{\pi \rho^2}$ 
		& $\displaystyle \frac{1+t-2(1-r)_+^{5}(8r^{2}+5r+1)}{4\pi (1-t^2)}$ \\\addlinespace\hline
		\addlinespace
		$\text{WE}_{3,2}$
		& $\displaystyle \frac{3(1-r)_+^{6}(35r^{2}+18r+3)}{\pi \rho^2}$ 
		& $\displaystyle \frac{1+t-2(1-r)_+^{7}(21r^{3}+19r^{2}+7r+1)}{4\pi (1-t^2)}$ \\
		\bottomrule
	\end{tabular}
\end{table}



\section{Vector field quasi-interpolation on the sphere}
\label{sec:quasi-interp}

This section develops our new quasi-interpolation method for vector field approximation of tangential vector fields on $\mathbb{S}^2$, which provides a direct approximation of the Helmholtz–Hodge decomposition without the computational cost associated with solving large linear systems as opposed to the existing kernel-based interpolation.

\subsection{Vector quasi-interpolation}
We begin with a scaled zonal kernel $\scaleKer$ derived from the restriction on $\mathbb{S}^2$ of a radial basis function on $\mathbb{R}^3$, satisfying Assumption \ref{assump:kernel}. We define the matrix- kernel $\Psi_\rho$ as the superposition of its divergence/curl-free components:
\begin{equation}\label{eq:Matrix_kernel}
	\Psi_\rho(x,y) = \Psi^\dive_\rho(x,y) + \Psi^\curl_\rho(x,y).
\end{equation}
For a target tangent vector field $\bff \in \bld{L}_2(\bbS^2)$, we define the following vector-valued integral operator $\calC_{\Psi_\rho}$ on $\bld{L}_2(\bbS^2)$:
\begin{equation}\label{eq:ConvOperator1}
	(\calC_{\Psi_\rho}\bff)(x) := (\Psi_\rho * \bff)(x) = \int_{\bbS^2} \Big(\Psi_\rho^{\dive}(x,y) + \Psi_\rho^{\curl}(x,y)\Big) \bff(y) \dmu(y).
\end{equation}
Here the integral is carried out componentwise.

\begin{definition}\label{multiplier-def}
	A linear operator $\mathcal{T}$ on $\bld{L}_2(\mathbb{S}^2)$ is called a Fourier-Legendre multiplier, if for all $\ell$ and $k$, there exist $a_{\ell,k}, b_{\ell,k} \in \R$, such that
	\[
	\widehat{\mathcal{T}(\mathbf{y}_{\ell,k})}(\ell,k) = a_{\ell,k}, \quad \widehat{\mathcal{T}(\mathbf{z}_{\ell,k})}(\ell,k) = b_{\ell,k}.
	\]
	A Fourier-Legendre multiplier is said to be simple if the numbers $a_{\ell,k}, b_{\ell,k}$ in the equations above are independent of $k.$
\end{definition}

\begin{proposition}\label{multiplier}
	The integral operator defined in \eqref{eq:ConvOperator1} is a simple Fourier-Legendre multiplier. 
\end{proposition}

\begin{proof} In \eqref{eq:ConvOperator1}, set $\bff =\mathbf{y}_{\ell',k'} $ for a prefixed pair $(\ell', k')$. By \eqref{eq:Four_divekernel} and \eqref{eq:Four_curlkernel},
	we write 
	\begin{align*}
		(\calC_{\Psi_\rho}\bff)(x) = &~\int_{\bbS^2} \Big(\Psi_\rho^{\dive}(x,y) + \Psi_\rho^{\curl}(x,y)\Big) \bff(y) \dmu(y)\\
		=&~\mathop{\mathlarger{\int}}_{\bbS^2}
		\left\{\sum_{\ell=1}^\infty \widehat{\psi}(\ell)\sum_{k=1}^{2\ell+1}\mathbf{y}_{\ell,k}(x) \mathbf{y}_{\ell,k}(y)^{\top}+\sum_{\ell=1}^\infty \widehat{\psi}(\ell)\sum_{k=1}^{2\ell+1}\mathbf{z}_{\ell,k}(x) \mathbf{z}_{\ell,k}(y)^{\top}\right\}\bff(y) \dmu(y)\\
		=	&~	\sum_{\ell=1}^\infty \widehat{\psi}(\ell)\sum_{k=1}^{2\ell+1}\mathbf{y}_{\ell,k}(x) \int_{\bbS^2}\mathbf{y}_{\ell,k}(y)^{\top}\mathbf{y}_{\ell',k'}(y)\dmu(y) \\
		& +\sum_{\ell=1}^\infty \widehat{\psi}(\ell)\sum_{k=1}^{2\ell+1}\mathbf{z}_{\ell,k}(x) \int_{\bbS^2}\mathbf{z}_{\ell,k}(y)^{\top}\mathbf{y}_{\ell',k'}(y)\dmu(y)\\
		=&~\widehat{\psi}(\ell')\ \mathbf{y}_{\ell',k'}(x).
	\end{align*}
	Here we have used the orthonormality of the vector spherical harmonics. Similarly, we show that $\calC_{\Psi_\rho}(\mathbf{z}_{\ell',k'} )(x) = \mathbf{z}_{\ell',k'}(x)$.
\end{proof}

We will show in Section 5 that
these multiplier operators are approximate identities while $\rho \rightarrow 0$ and that the approximation power is  determined by the scaling parameter $\rho$ and a certain asymptotic property of
the Fourier-Legendre coefficients of the kernel employed therein.
\begin{definition}[\emph{QMC designs}]
	Let $N \in \mathbb{N}$ and $\sigma>1$. A sequence
	$X_N:=\{x_j\}_{j=1}^{N} \subset \bbS^2$  with $N\rightarrow\infty$ is called a sequence of \emph{QMC designs} for the Sobolev space $H^{\sigma}(\bbS^2)$  if there exists a constant $c_\sigma>0$, independent of $N$, such that
	\begin{equation}\label{eq:QMC_estimate}
		\sup_{f\in H^{\sigma}(\bbS^2)}\left|\int_{\bbS^2} f(x)\dmu(x)-\calI_{X_N}(f)\right| \leq c_\sigma N^{-\frac{\sigma}{2}}\|f\|_{H^\sigma(\bbS^2)},\quad \calI_{X_N}(f):=\frac{4\pi}{N}\sum_{j=1}^N \bff(x_j).
	\end{equation}
\end{definition} 

This definition was introduced by Brauchart et al. \cite{Brauchart_2014MCoM_qmc}. Among various spherical quadrature rules developed in the literature \cite{Brauchart_2007ConApp_numerical,Hesse_2010Incollection_numerical,Sloan_2004AdvCM_extremal}, we employ exclusively QMC  designs to discretize the integral in \eqref{eq:ConvOperator1} 
en route to getting a vector-valued quasi-interpolant $\calL_{\Psi_\rho}$: 
\begin{equation}\label{eq:QMCQI}
	(\calL_{\Psi_\rho}\bff)(x) := \frac{4\pi}{N}\sum_{j=1}^{N}\Psi_\rho(x,x_j) \bff(x_j), \quad x \in \bbS^2.
\end{equation}
We remind an interested reader that a quasi-interpolant of the present form depends on an $N$-point set $X_N:=\{x_j\}_{j=1}^{N} \subset \bbS^2$ based on which we discretize the integral in \eqref{eq:ConvOperator1}. Because of the exclusivity of using QMC designs (in discretizing the underlying integral), we choose the neutral symbol $\calL_{\Psi_\rho}\bff$ to denote such an quasi-interpolant.


\subsection{Approximation of the Helmholtz–Hodge decomposition}
A significant advantage of the construction in \eqref{eq:ConvOperator1} is that it allows for the component-wise approximation of the Helmholtz–Hodge decomposition. Let $\calP_\dive$ and $\calP_\curl$ denote the orthogonal projections onto the divergence-free and curl-free subspaces of $\bld{L}_2(\bbS^2)$, respectively. The unique decomposition is given by
$$
\bff = \bff_{\dive} + \bff_{\curl}, \quad \text{where } \bff_{\dive} := \calP_\dive\bff, \quad \bff_{\curl} := \calP_\curl\bff.
$$
Due to the orthogonality of the vector spherical harmonics, the convolution operator $\calC_{\Psi_\rho}$ respects this decomposition. Specifically, the divergence/curl-free components of the approximation are given by
\begin{equation}\label{eq:ConvOperator2}
	(\calC_{\Psi_\rho}\bff)_{\dive} = \int_{\bbS^2}\Psi_\rho^{\dive}(x,y) \bff(y)\dmu(y), \quad
	(\calC_{\Psi_\rho}\bff)_{\curl} = \int_{\bbS^2}\Psi_\rho^{\curl}(x,y) \bff(y)\dmu(y).
\end{equation}
Discretizing the two integrals above, we obtain the divergence/curl-free components of the 
quasi-interpolant:
\begin{equation}\label{eq:QMCQI_components}
	(\calL_{\Psi_\rho}\bff)_{\dive} = \frac{4\pi}{N}\sum_{j=1}^{N}\Psi_\rho^{\dive}(x,x_j) \bff(x_j), \quad
	(\calL_{\Psi_\rho}\bff)_{\curl} = \frac{4\pi}{N}\sum_{j=1}^{N}\Psi_\rho^{\curl}(x,x_j) \bff(x_j).
\end{equation}
We use \eqref{eq:proposed_kernels} to rewrite \eqref{eq:QMCQI_components} in the following more executable format:
\begin{equation}\label{eq:explicit_div}
	(\calL_{\Psi_\rho}\bff)_{\dive}(x) = \frac{4\pi}{N}\sum_{j=1}^{N}\Biggl\{-\kappa'(x \cdot x_j)\Bigl[(x\times x_j)^\top \bff(x_j)\Bigr](x\times x_j) + \kappa(x \cdot x_j)\Bigl[(x \cdot x_j) \bff(x_j) - \big(x^\top \bff(x_j)\big)x_j\Bigr]\Biggr\};
\end{equation}
\begin{equation}\label{eq:explicit_curl}
	(\calL_{\Psi_\rho}\bff)_{\curl}(x) = \frac{4\pi}{N}\sum_{j=1}^{N}\Biggl\{\kappa'(x \cdot x_j)\big(x^\top\bff(x_j)\big)(x_j-(x \cdot x_j)x) + \kappa(x \cdot x_j)\Bigl[\bff(x_j) - \big(x^\top \bff(x_j)\big)x\Bigr]\Biggr\}.
\end{equation}

We emphasize here that with the availability of the vector-valued data $\bff(x_j)$ one can simply write down the desired quasi-interpolants as in 
Equations \eqref{eq:explicit_div} and \eqref{eq:explicit_curl}. In particular, there is no deed to evaluate the convolution integrals when implementing a quasi-interpolation algorithm. 
This contrasts sharply to the existing vector field interpolation method in which one must  invert a large and dense interpolation matrix. In the next section, we will investigate 
the approximation power and robustness of the quasi-interpolation method.


\section{Approximation Power and Robustness}
\label{sec:err}
We break the section into three subsections in which we discuss successively (1) Approximation properties of multiplier operators; (2) Quadrature errors; (3) Robustness of quasi-interpolation.


\subsection{Approximation properties of the convolution operator}

\begin{theorem}\label{thm:ConvErr_L2}
	Let $\bff \in \bH^{\sigma}(\bbS^2)$ be a tangent vector field, and let $\Psi_\rho$ be the matrix-valued kernel defined in \eqref{eq:Matrix_kernel} with the zonal kernel $\psi_{\rho}$ satisfying Assumption \ref{assump:kernel} with scaling parameter $0 < \rho <  1$ and $m\geq \sigma$. Then, for any $0 \leq \tau \leq \sigma$, there exists a constant $C > 0$, independent of $\rho$ and $\bff$, such that
	\begin{equation}\label{eq:sobolev_bound}
		\|\bff - \calC_{\Psi_\rho}\bff\|_{\bH^{\tau}(\bbS^2)} \leq C \rho^{\sigma-\tau} \|\bff\|_{\bH^{\sigma}(\bbS^2)}.
	\end{equation}
\end{theorem}

\begin{proof}
	Recall that the eigenvalues of the vector-form Laplacian satisfy $\lambda_\ell = \ell(\ell+1) \sim \ell^2$. Using the Fourier series expansion of $\calC_{\Psi_\rho}\bff$, the squared Sobolev error can be written as
	\begin{equation}\label{eq:error_sum}
		\|\bff - \calC_{\Psi_\rho}\bff\|_{\bH^{\tau}(\bbS^2)}^2 = \sum_{\ell=1}^\infty \sum_{k=1}^{2\ell+1} (1+\lambda_\ell)^{\tau} |1-\hatphi(\ell)|^2 \left( |\hatdiv{\bff}|^2 + |\hatcurl{\bff}|^2 \right).
	\end{equation}
	We split the summation into two parts: $\ell \leq \ell_\rho$ and $\ell > \ell_\rho$ respectively, where $\ell_\rho \sim \rho^{-1}$ is the cut-off frequency associated with the scaling parameter.
	
	For the first part, we apply Assumption \ref{assump:kernel}, which guarantees $|1-\hatphi(\ell)| \leq C (\rho \ell)^{\sigma}$. Thus,
	\begin{align*}
		\mathcal{E}_{1} &:= \sum_{\ell \leq \ell_\rho} (1+\lambda_\ell)^{\tau} |1-\hatphi(\ell)|^2 \sum_{k=1}^{2\ell+1} \left( |\hatdiv{\bff}|^2 + |\hatcurl{\bff}|^2 \right) \\
		&\leq C \rho^{2\sigma} \sum_{\ell \leq \ell_\rho} \ell^{2\sigma} (1+\lambda_\ell)^{\tau} \sum_{k=1}^{2\ell+1} \left( |\hatdiv{\bff}|^2 + |\hatcurl{\bff}|^2 \right).
	\end{align*}
	Multiplying and dividing by $(1+\lambda_\ell)^{\sigma-\tau}$, we obtain
	\begin{equation*}
		\begin{aligned}
			\mathcal{E}_{1} \leq &~ C \rho^{2(\sigma-\tau)} \sum_{\ell \leq \ell_\rho} \underbrace{(\rho \ell)^{2\tau} \left(\frac{\ell^2}{1+\lambda_\ell}\right)^{\sigma-\tau}}_{\leq C} (1+\lambda_\ell)^{\sigma}  \sum_{k=1}^{2\ell+1} \left( |\hatdiv{\bff}|^2 + |\hatcurl{\bff}|^2 \right) \\
			\leq &~ C \rho^{2(\sigma-\tau)} \|\bff\|_{\bH^{\sigma}(\bbS^2)}^2,
		\end{aligned}
	\end{equation*}
	where we used the fact that $\rho \ell \leq C$ for $\ell \leq \ell_\rho$.
	
	For the second part, we utilize the uniform boundedness of the Fourier-Legendre coefficients, $|\hatphi(\ell)| \leq c_2$. Specifically, since $\ell > \ell_\rho$, we have $(1+\lambda_\ell)^{\tau-\sigma} \leq C \ell^{2(\tau-\sigma)} \leq C \rho^{2(\sigma-\tau)}$. Therefore,
	\begin{align*}
		\mathcal{E}_{2} &:= \sum_{\ell > \ell_\rho} (1+\lambda_\ell)^{\tau} |1-\hatphi(\ell)|^2 \sum_{k=1}^{2\ell+1} \left( |\hatdiv{\bff}|^2 + |\hatcurl{\bff}|^2 \right) \\
		&\leq 2(1+c_2^2) \sum_{\ell > \ell_\rho} (1+\lambda_\ell)^{\tau-\sigma} (1+\lambda_\ell)^{\sigma} \sum_{k=1}^{2\ell+1} \left( |\hatdiv{\bff}|^2 + |\hatcurl{\bff}|^2 \right) \\
		&\leq C \rho^{2(\sigma-\tau)} \|\bff\|_{\bH^{\sigma}(\bbS^2)}^2.
	\end{align*}
	Combining the estimates for $\mathcal{E}_{1}$ and $\mathcal{E}_{2}$ yields the desired result.
\end{proof}

An immediate consequence of Theorem \ref{thm:ConvErr_L2} is that $\calC_{\Psi_\rho}$ acts as an approximate projection operator. 

\begin{corollary}\label{cor:approx_proj}
	Under the assumptions of Theorem \ref{thm:ConvErr_L2}, we have
	\begin{equation*}
		\|\calC_{\Psi_\rho}\bff - \calC_{\Psi_\rho}^2\bff\|_{\bH^{\tau}(\bbS^2)} \leq C \rho^{\sigma-\tau} \|\bff\|_{\bH^{\sigma}(\bbS^2)}.
	\end{equation*}
\end{corollary}

\begin{proof}
	Making use of Proposition \ref{multiplier}, we identify the operator $\calC_{\Psi_\rho}$ with its simple multiplier $\hatphi(\ell)$. Thus,  the operator $\calC_{\Psi_\rho} - \calC_{\Psi_\rho}^2$ is identifiable with the simple multiplier $\hatphi(\ell) - \hatphi(\ell)^2$. Furthermore, we have
	\begin{equation*}
		|\hatphi(\ell) - \hatphi(\ell)^2| \leq C |1-\hatphi(\ell)|.
	\end{equation*}
	Here we have used the fact that the Fourier-Legendre coefficients $\hatphi(\ell)$ are uniformly bounded with respect to  $\ell$. The rest of the proof  proceeds identically to that of Theorem \ref{thm:ConvErr_L2}.
\end{proof}

It follows from \eqref{eq:ConvOperator1} and \eqref{eq:ConvOperator2}  that the operator $\calC_{\Psi_\rho}$ preserves the Helmholtz–Hodge decomposition. This leads us directly to  the same error estimate for the divergence/curl-free components of a vector field.


\begin{corollary}
	Under the assumptions of Theorem \ref{thm:ConvErr_L2}, the following error estimates hold:
	$$\|\bff_{\dive}-(\calC_{\Psi_\rho}\bff)_{\dive}\|_{\bH^{\tau}(\bbS^2)}\leq C\rho^{\sigma-\tau}\|\bff\|_{\bHsig},$$
	$$\|\bff_{\curl}-(\calC_{\Psi_\rho}\bff)_{\curl}\|_{\bH^{\tau}(\bbS^2)}\leq C\rho^{\sigma-\tau}\|\bff\|_{\bHsig}.$$
\end{corollary}

The rest of this subsection devotes to the derivation of  an $\bld{L}_\infty$-error estimate for the convolution operator. 

\begin{theorem}\label{thm:Linf_error}
	Let $\bff \in \bH^{\sigma}(\bbS^2)$ with $\sigma > 1$. Suppose the kernel $\Psi_\rho$ satisfies Assumption \ref{assump:kernel} with $m = \sigma$. Then, there exists a constant $C > 0$ independent of $\rho$ and $\bff$,     
	such that
	\begin{equation}\label{eq:Linf_bound}
		\|\bff - \mathcal{C}_{\Psi_\rho}\bff\|_{\bld{L}_\infty(\bbS^2)} \leq C \rho^{\sigma-1} \|\bff\|_{\bH^{\sigma}(\bbS^2)}.
	\end{equation}
\end{theorem}

\begin{proof}
	Since $\sigma > 1$, the Sobolev embedding $\bH^{\sigma}(\bbS^2) \hookrightarrow \bC^0(\bbS^2)$ holds, ensuring the pointwise error is well-defined. Let $\mathcal{E}_\rho \bff(x) := \bff(x) - \mathcal{C}_{\Psi_\rho}\bff(x)$. Using the vector spherical harmonic expansion, we have
	\begin{equation*}
		\mathcal{E}_\rho \bff(x) = \sum_{\ell=1}^\infty (1 - \hatphi(\ell)) \sum_{k=1}^{2\ell+1} \left[ \widehat{\bff}_{\ell,k}^{\text{div}}\  \by_{\ell,k}(x) + \widehat{\bff}_{\ell,k}^{\text{curl}}\ \bz_{\ell,k}(x) \right],
	\end{equation*}
	where $\widehat{\bff}_{\ell,k}^{\text{div}}$ and $\widehat{\bff}_{\ell,k}^{\text{curl}}$ denote the divergence-free and curl-free Fourier coefficients, respectively. Applying the Cauchy-Schwarz inequality over the summation indices $\ell$ and $k$ yields
	\begin{equation}\label{eq:CS_split}
		|\mathcal{E}_\rho \bff(x)|^2 \leq \left( \sum_{\ell=1}^\infty \mathcal{K}_\ell(x) \right) \|\bff\|_{\bH^\sigma(\bbS^2)}^2,
	\end{equation}
	where the kernel function $\mathcal{K}_\ell(x)$ is given by
	\begin{equation*}
		\mathcal{K}_\ell(x) := (1+\lambda_\ell)^{-\sigma} |1 - \hatphi(\ell)|^2 \sum_{k=1}^{2\ell+1} \left( |\by_{\ell,k}(x)|^2 + |\bz_{\ell,k}(x)|^2 \right).
	\end{equation*}
	The addition theorem for vector spherical harmonics (see, e.g., \cite{Freeden_1993MMAS_vector,Freeden_2008book_spherical,Mueller_1966book_spherical})  states that
	\begin{equation}\label{eq:addition_thm}
		\sum_{k=1}^{2\ell+1} \left( |\by_{\ell,k}(x)|^2 + |\bz_{\ell,k}(x)|^2 \right) = \frac{2\ell+1}{2\pi}.
	\end{equation}
	Thus, the proof of the theorem is reduced  to showing that	\begin{equation}\label{eq:series_bound}
		S_\rho := \sum_{\ell=1}^\infty \ell (1+\ell^2)^{-\sigma} |1 - \hatphi(\ell)|^2 \le C \rho^{2(\sigma-1)}.
	\end{equation}
	To this goal, we write $S_\rho=S_{\rho, 1}+S_{\rho, 2,}$, in which
	\[
	S_{\rho, 1}:= \sum_{\ell=1}^{\ell_\rho} \ell (1+\ell^2)^{-\sigma} |1 - \hatphi(\ell)|^2, \quad S_{\rho, 2}:= \sum_{\ell=\ell_\rho+1}^\infty \ell (1+\ell^2)^{-\sigma}|1 - \hatphi(\ell)|^2. 
	\]
	By \eqref{Assumption-1}, we have $|1 - \hatphi(\ell)| \leq C (\rho \ell)^\sigma$ for $\ell \le \ell_\rho$. It then follows that
	\begin{align*}
		S_{\rho, 1} 
		\leq C \rho^{2\sigma} \sum_{\ell=1}^{\ell_\rho} \ell^{1-2\sigma} \ell^{2\sigma} \leq C \rho^{2\sigma} \ell_\rho^2 \leq C \rho^{2(\sigma-1)}.
	\end{align*}
	
	For $\ell > \ell_\rho$, $|1 - \hatphi(\ell)|$ is uniformly bounded with respect to $\ell$ by \eqref{Assumption-2}. Thus we have
	\begin{align*}
		S_{\rho, 2} & \le C\sum_{\ell=\ell_\rho+1}^\infty \ell  (1+\ell^2)^{-\sigma} 
		\leq C \sum_{\ell=\ell_\rho+1}^\infty \ell^{1-2\sigma}\leq C \ell_\rho^{2-2\sigma} \leq C \rho^{2\sigma-2}.
	\end{align*}
	This completes the proof.
\end{proof}

\subsection{Quadrature error}
We prove two lemmas prior to presenting our main result (Theorem \ref{thm:QMC_error}).

\begin{lemma}\label{lem:ExpressionQMCQI}
	Let $X_N$ be a QMC design. Let $\calL_{\Psi_\rho}$ be the vector field quasi-interpolant based on $X_N$. Then $\calL_{\Psi_\rho}$ admits the following 
	Fourier-Legendre expansion: 
	\begin{equation*}
		\calL_{\Psi_\rho}\bff(x)
		=\sum_{\ell=1}^\infty \hatphi(\ell) \sum_{k=1}^{2\ell+1} \left[\calI_{X_N}(\by^{\top}_{\ell,k}\bff)\by_{\ell,k}(x)+\calI_{X_N}(\bz^{\top}_{\ell,k}\bff)\bz_{\ell,k}(x)\right].
	\end{equation*}
\end{lemma}
\begin{proof}
	We write by Theorem~\ref{thm:Four_expan_divcurl} and Equation \eqref{eq:QMCQI} that 
	\begin{align*}
		&\calL_{\Psi_\rho}\bff(x) = \frac{4\pi}{N}\sum_{j=1}^{N} \left\{ \sum_{\ell=1}^\infty \hatphi(\ell)\sum_{k=1}^{2\ell+1}\Big[\by_{\ell,k}(x) \by_{\ell,k}(x_j)^{\top}\bff(x_j)+\bz_{\ell,k}(x) \bz_{\ell,k}(x_j)^{\top}\bff(x_j)\Big] \right\} \\
		= &  \sum_{\ell=1}^\infty \hatphi(\ell)\left\{\sum_{k=1}^{2\ell+1}\left[\by_{\ell,k}(x) \left(\frac{4\pi}{N}\sum_{j=1}^{N}\by_{\ell,k}(x_j)^{\top}\bff(x_j)\right)+\bz_{\ell,k}(x) \left(\frac{4\pi}{N}\sum_{j=1}^{N}\bz_{\ell,k}(x_j)^{\top}\bff(x_j)\right)\right] \right\} \\
		=&\sum_{\ell=1}^\infty \hatphi(\ell) \sum_{k=1}^{2\ell+1} \left[\calI_{X_N}(\by^{\top}_{\ell,k}\bff)\by_{\ell,k}(x)+\calI_{X_N}(\bz^{\top}_{\ell,k}\bff)\bz_{\ell,k}(x)\right],	\end{align*}
	which is the desired result of the lemma.
\end{proof}

\begin{lemma}\label{lem:QMC_quaderr}
	Let $\sigma > 1$ and $X_N$ a QMC design 
	for $H^\sigma(\bbS^2)$. Then there exists a constant $C > 0$ independent of $N$ such that
	\begin{equation*}
		\Big|\langle \bff, \bgg  \rangle_{\bld{L}_2} - \calI_{X_N}( \bff^{\top}\bgg)\Big| \leq C N^{-\frac{\sigma}{2}} \|\bff\|_{\bHsig} \|\bgg\|_{\bHsig}, \quad \forall\ \bff, \bgg \in \bHsig.
	\end{equation*} 
	
\end{lemma}

\begin{proof}
	Note first that the condition  $\sigma > 1$ implies that
	the Sobolev space $H^\sigma(\bbS^2)$ is a Banach algebra under pointwise addition and multiplication of functions. We then have
	\begin{equation}\label{eq:algebra_prop}
		\|\bff^\top \bgg\|_{H^\sigma(\bbS^2)} \leq  C\ \|\bff\|_{\bHsig} \|\bgg\|_{\bHsig}.
	\end{equation}
	It follows from \eqref{eq:QMC_estimate} that
	\begin{align*}
		\Big|\int_{\bbS^2} \bff^\top \bgg \dmu(x)-\calI_{X_N}(\bff^\top \bgg)\Big|
		\leq c_{\sigma} N^{-\frac{\sigma}{2}} \|\bff^\top \bgg\|_{H^{\sigma}(\bbS^2)}
		\leq C\ N^{-\frac{\sigma}{2}} \|\bff\|_{\bHsig} \|\bgg\|_{\bHsig},
	\end{align*}
	which completes the proof.
\end{proof}

\begin{theorem}\label{thm:QMC_error}
	Let $\sigma > 1$ and $X_N$ a QMC design 
	for $H^\sigma(\bbS^2)$. Suppose that a scaled zonal kernel $\psi_\rho$ satisfies Assumption \ref{assump:kernel} with scale parameter $0<\rho<1$ and order $m=\sigma$. Let  $0 \leq \tau \leq \sigma$ and $s = \tau + \sigma$. Then there exists a constant $C>0$ independent of $\rho$,  $N$, and $ \bff \in \bHsig,$ such that
	\begin{equation*}
		\|\bff-\calL_{\Psi_\rho} \bff\|_{\bHtau} \leq C \left(\rho^{\sigma-\tau} + \rho^{-s} N^{-\frac{\sigma}{2}}\right) \|\bff\|_{\bHsig}.
	\end{equation*}
\end{theorem}

\begin{proof}
	We first write by the triangle inequality that
	\begin{equation}\label{eq:err_decomp}
		\|\bff-\calL_{\Psi_\rho} \bff\|_{\bHtau} \leq \|\bff-\calC_{\Psi_\rho}\bff\|_{\bHtau} + \|\calC_{\Psi_\rho}\bff-\calL_{\Psi_\rho} \bff\|_{\bHtau}.
	\end{equation}
	We bound the first term making use of  the  error estimate in Theorem \ref{thm:ConvErr_L2} as follows.
	\begin{equation}\label{eq:conv_bound}
		\|\bff-\calC_{\Psi_\rho}\bff\|_{\bHtau} \leq C \rho^{\sigma-\tau} \|\bff\|_{\bH^{\sigma}(\bbS^2)} \leq C \rho^{\sigma-\tau} \|\bff\|_{\bHsig}.
	\end{equation}
	To bound the second term,  Let $\be_N := \calC_{\Psi_\rho}\bff-\calL_{\Psi_\rho} \bff$. We employ Parseval's identity and the result of Lemma \ref{lem:ExpressionQMCQI} to derive
	\begin{equation*}
		\|\be_N\|_{\bHtau}^2 = \sum_{\ell=1}^{\infty} (1+\lambda_\ell)^{\tau} |\widehat{\psi}_\rho(\ell)|^2 \sum_{k=1}^{2\ell+1} \left( \left|\calE_{\ell,k}^{\by}\right|^2 + \left|\calE_{\ell,k}^{\bz}\right|^2 \right),
	\end{equation*}
	where $\calE_{\ell,k}^{\by} := \langle \bff, \by_{\ell,k} \rangle - \calI_{X_N}(\by_{\ell,k}^\top \bff)$, and  $\calE_{\ell,k}^{\bz}:= \langle \bff, \bz_{\ell,k} \rangle - \calI_{X_N}(\bz_{\ell,k}^\top \bff)$ . Applying Lemma \ref{lem:QMC_quaderr}, we get that
	\begin{equation*}
		\max \left\{|\calE_{\ell,k}^{\by}|, |\calE_{\ell,k}^{\bz}|\right\} \leq C N^{-\frac{\sigma}{2}} \|\by_{\ell,k}\|_{\bHsig} \|\bff\|_{\bHsig} = C N^{-\frac{\sigma}{2}} (1+\lambda_\ell)^{\frac{\sigma}{2}} \|\bff\|_{\bHsig}.
	\end{equation*}
	It then follows that
	\begin{align*}
		\|\be_N\|_{\bHtau}^2 
		&\leq C N^{-\sigma} \|\bff\|_{\bHsig}^2 \sum_{\ell=1}^{\infty}\sum_{k=1}^{2\ell+1} (1+\lambda_\ell)^{\tau+\sigma} |\widehat{\psi}_\rho(\ell)|^2   \\
		&= C N^{-\sigma} \|\bff\|_{\bHsig}^2 \|\psi_\rho\|_{H^s(\bbS^2)}^2 \le C\rho^{-s} N^{-\sigma} \|\bff\|_{\bHsig}^2.
	\end{align*}
	Here in the last inequality above, we have used the result of Lemma \ref{lem:Equiv_Sobolev_Native} ( $\|\psi_\rho\|_{H^s(\bbS^2)} \leq C \rho^{-s}$).  Combining this with \eqref{eq:err_decomp} and \eqref{eq:conv_bound} completes the proof.
\end{proof}

It is well-known that the divergence/curl-free components in the 
Helmholtz-Hodge decomposition of an $\bld{L}_2(\bbS^2)$-vector-field 
are mutually orthogonal
in $\bld{L}_2(\bbS^2)$. In the same vein, they are also orthogonal in  $\bHtau$ for each $\tau \ge 0.$
Furthermore, the orthogonality is preserved under the actions of the convolution operator $\calC_{\Psi_\rho}$  and the quasi-interpolation operator $\calL_{\Psi_\rho}$.  Consequently, we have the following ``error-decomposition":
\begin{equation*}
	\|\bff-\calL_{\Psi_\rho} \bff\|_{\bHtau}^2 = \|\bff_{\dive}-(\calL_{\Psi_\rho}\bff)_{\dive}\|_{\bHtau}^2 + \|\bff_{\curl}-(\calL_{\Psi_\rho}\bff)_{\curl}\|_{\bHtau}^2,
\end{equation*}
and the respective error estimates for the divergence/curl-free components in the 
Helmholtz-Hodge decomposition of an $\bld{L}_2(\bbS^2)$-vector-field.


\begin{corollary}\label{cor:HHD_error}
	Let $\bff \in \bld{L}_2(\bbS^2)$ and $\bff = \bff_{\dive} + \bff_{\curl}$ be the Helmholtz-Hodge decomposition of $\bff$. Under the assumptions of Theorem \ref{thm:QMC_error}, we have
	\begin{align*}
		\|\bff_{\dive}-(\calL_{\Psi_\rho}\bff)_{\dive}\|_{\bH^{\tau}(\bbS^2)} &\leq C\left(\rho^{\sigma-\tau}  +\rho^{-s}N^{-\frac{\sigma}{2}}\right)\|\bff\|_{\bHsig}, \\
		\|\bff_{\curl}-(\calL_{\Psi_\rho}\bff)_{\curl}\|_{\bH^{\tau}(\bbS^2)} &\leq C\left(\rho^{\sigma-\tau}  +\rho^{-s}N^{-\frac{\sigma}{2}}\right)\|\bff\|_{\bHsig}.
	\end{align*}
\end{corollary}

\subsection{Robustness of the quasi-interpolation algorithm}
Let $X_N:=\{x_1,\ldots,x_N\}$ be a QMC design for $H^\sigma(\bbS^2)$.
Suppose that the measurements of a vector field $\bff$ on $X_N$ are contaminated and result in the observations $\bff(x_j)+\beps_j$ ( $1\leq j\leq N$), where $\beps_j$ are noise vectors. The corresponding 
vector-field quasi-interpolant is then of the form:
\begin{equation}\label{eq:QMCQI_noise}
	\calL_{\Psi_\rho}^{\epsilon}\bff(x)=\frac{4\pi}{N}\sum_{j=1}^{N}\Psi_\rho(x,x_j)(\bff(x_j)+\beps_j) ,\quad x\in\bbS^d.
\end{equation}
For the sake of applications, we make the following assumption on the noise vectors $\beps_j$ ( $1\leq j\leq N$).

\begin{assumption}\label{noise-assump}
	(i) The random vectors   $\beps_1, \cdots, \beps_N$  are independent and $\bbE \beps_j=0$ for all $j=1,\ldots,N$: (ii) There exists a constant $B$ independent of $N$, such that ${\rm tr} (B_j) \le B, \; \forall\ 1\leq j\leq N,$ where $B_j$ is the covariance matrix for the three component-random-variables of $\beps_j$. 
\end{assumption}

The following theorem solidifies the robustness of our quasi-interpolation algorithm against random noises.
\begin{theorem}
	Suppose that in \eqref{eq:QMCQI_noise} the random noise vectors $\beps_j$ satisfy Assumption \ref{noise-assump}. Under the assumptions of Theorem \ref{thm:QMC_error} for the case $\tau=0$, we have that
	\begin{equation*}
		\bbE\Big[\big\|\calL_{\Psi_\rho}^{\epsilon}\bff-\bff\big\|_{\bld{L}_2(\bbS^2)}\Big] \leq C\left(\rho^\sigma+\rho^{-\sigma}N^{-\frac{\sigma}{2}}\right)\|\bff\|_{\bHsig} + C\rho^{-1}N^{-\frac{1}{2}}.
	\end{equation*}
\end{theorem}

\begin{proof} We break the expected $\bld{L}_2$-error in two parts in the order of stochastic error and deterministic error as follows.
	\[
	\big\|\calL_{\Psi_\rho}^{\epsilon}\bff-\bff\big\|_{\bld{L}_2(\bbS^2)} \le \big\|\calL_{\Psi_\rho}^{\epsilon}\bff-\calL_{\Psi_\rho}\bff\big\|_{\bld{L}_2(\bbS^2)} + \big\|\calL_{\Psi_\rho}\bff-\bff\big\|_{\bld{L}_2(\bbS^2)}.
	\]
	We use Theorem \ref{thm:QMC_error} for the case $\tau=0$ to derive the deterministic error estimate:
	\begin{equation*}
		\big\|\calL_{\Psi_\rho}\bff-\bff\big\|_{\bld{L}_2(\bbS^2)} \leq C\left(\rho^\sigma+\rho^{-\sigma}N^{-\frac{\sigma}{2}}\right)\|\bff\|_{\bHsig}.
	\end{equation*}
	The rest of the proof is devoted to the stochastic error estimate. 
	Let $\boldsymbol{\eta} := \calL_{\Psi_\rho}^{\epsilon}\bff - \calL_{\Psi_\rho}\bff$. Using Jensen's inequality $\bbE\|\boldsymbol{\eta}\|_{\bld{L}_2} \leq (\bbE\|\boldsymbol{\eta}\|_{\bld{L}_2}^2)^{1/2}$ followed by exchanging the order of integrals, we have 
	\begin{align*}
		\bbE\|\boldsymbol{\eta}\|_{\bld{L}_2(\bbS^2)}^2 
		&= \frac{(4\pi)^2}{N^2} \int_{\bbS^2} \bbE \Bigl\| \sum_{j=1}^{N} \Psi_\rho(x,x_j)\beps_j \Bigr\|^2_{\bld{L}_2(\bbS^2)} \dmu(x) \\
		&= \frac{(4\pi)^2}{N^2} \int_{\bbS^2} \sum_{j,k=1}^N \bbE \Big[ \beps_j^\top \Psi_\rho(x,x_j)^\top \Psi_\rho(x,x_k)\beps_k \Big] \dmu(x)\\
		&=\frac{(4\pi)^2}{N^2} \sum_{j=1}^N \int_{\bbS^2} \bbE \|\Psi_\rho(x,x_j)\beps_j\|_2^2 \dmu(x).
	\end{align*}
	In the last equation above, we have used the equations
	$\bbE[\beps_j \beps_k^\top] = 0$ if $j \ne k$, which hold true because the random vectors $\beps_1, \ldots, \beps_N$ are independent.  Utilizing the 
	inequalities $\bbE \|\Psi_\rho(x,x_j)\beps_j\|_2^2 \le \bbE \left(\|\beps_j\|_2^2\right)\|\Psi_\rho(x,x_j)\|_2^2 = {\rm tr} (B_j) \|\Psi_\rho(x,x_j)\|_2^2 \le B \|\Psi_\rho(x,x_j)\|_2^2, $ and 
	$\|\Psi_\rho(x,x_j)\|_2^2 \le C\|\psi_\rho(\cdot,x_j)\|^2_{L_2(\bbS^2)}\leq C\rho^{-2}, \; \forall\ 1 \le j \le N;$ see \cite[Lem.~2.5]{Sun_2025_monte}, we derive
	\begin{equation*}
		\bbE\|\boldsymbol{\eta}\|_{\bld{L}_2(\bbS^2)}^2 \leq \frac{C}{N^2} \|\psi_\rho(\cdot,x_j)\|_{L_2(\bbS^2)}^2 \leq \frac{C}{N^2} \rho^{-2},
	\end{equation*}
	which implies that
	\[
	\bbE \big\|\calL_{\Psi_\rho}^{\epsilon}\bff-\calL_{\Psi_\rho}\bff\big\|_{\bld{L}_2(\bbS^2)}\le C\rho^{-1}N^{-\frac{1}{2}},
	\]
	and therefore completes the proof.
	
\end{proof}

\bigskip

\section{Numerical examples}
\label{sec:numer_exam}

In this section, we validate the theoretical Sobolev error estimates derived for the vector quasi-interpolant and its Helmholtz--Hodge decomposition. We assess convergence rates using two benchmark vector fields characterized by varying degrees of Sobolev regularity. Each field $\bff$ is generated explicitly from a scalar stream function $s$ and a velocity potential $v$ via the representation:
$$\bff = \bff_\dive + \bff_\curl = \mathbf{L}_* s + \nabla_* v,$$
which allows for an exact separation into divergence-free and curl-free components. The construction of these fields follows the benchmarks established in \cite{Fuselier_2009SINUM_stability} and implemented in \cite{LeGia_2021ToMS_algorithm}. The two test fields are defined as follows:
\begin{itemize}
	\item \textbf{Field 1.} The components are defined by spherical harmonics:
	\begin{align*}
		s_1(x) &= -\frac{1}{\sqrt{3}} Y_{1,0}(x) + \frac{8\sqrt{2}}{3\sqrt{385}} Y_{5,4}(x), \\
		v_1(x) &= \frac{1}{25} \left[ Y_{4,0}(x) + Y_{6,-3}(x) \right].
	\end{align*}
	\item \textbf{Field 2.} We retain the stream function from Field 1 ($s_2 = s_1$) but employ a linear combination of cubic B-splines for the potential:
	\begin{align*}
		v_2(x) &= \tfrac{1}{8}f(x; 5, \tfrac{\pi}{6}, 0) - \tfrac{1}{7}f(x; 3, \tfrac{\pi}{5}, -\tfrac{\pi}{7}) \\
		&\quad + \tfrac{1}{9}f(x; 5, -\tfrac{\pi}{6}, \tfrac{\pi}{2}) - \tfrac{1}{8}f(x; 3, -\tfrac{\pi}{5}, \tfrac{\pi}{3}),
	\end{align*}
	where $f=f(x;x_c,\zeta)$ is the cubic B-spline centered at $x_c$ (with spherical coordinates $(\theta_c, \lambda_c)$) and support parameter $\zeta$.
	
\end{itemize}

The two test fields exhibit distinct regularity properties. Field 1 is smooth, with $\bff_1 \in \bld{C}^\infty(\mathbb{S}^2)$. Field 2 has limited regularity governed by the cubic B-spline $f \in H^\tau(\mathbb{S}^2)$ for $\tau < 4$; consequently, the total field satisfies $\bff_2 \in \bH^\sigma(\mathbb{S}^2)$ for $\sigma < 3$. For the numerical experiments, we employ two classes of zonal kernels: restricted Gaussian and restricted $\text{WE}_{3,2}$ (see Table 1). Both base kernels satisfy Assumption 2.1 with $m=2$. To achieve higher orders ($m=4, 6, 8$), we apply the technique described in \cite{Farrell_2017IMAJNA_multilevel,Sun_2025IMAJNA_spherical}. The higher-order kernels $\psi_m(t;\rho)$ are constructed as linear combinations of the scaled base kernel $\psi_2(t;\rho)$:
\begin{equation}
	\psi_m(t;\rho) = \sum_{j=1}^K c_j \psi_{2}(t;\sqrt{a_j}\rho), \quad \text{with } c_j = \prod_{\substack{k=1,k\neq j}}^K \frac{a_k}{a_k-a_j}.
\end{equation}
The scaling parameters are chosen as $\{a_j\}_{j=1}^2 = \{1/2, 2/3\}$ for $m=4$, $\{a_j\}_{j=1}^3 = \{1/3, 2/3, 1\}$ for $m=6$, and $\{a_j\}_{j=1}^4 = \{1/4, 1/2, 3/4, 1\}$ for $m=8$.

\subsection{Convergence tests}
Our numerical investigations utilize the symmetric spherical $t$-designs (STDs) and maximum determinant (MD) points, which can be found in \cite{Womersley_misc_interpolation}. These point configurations are chosen for their quasi-uniformity and excellent quadrature properties, though we note that other distributions, such as minimum energy points \cite{Womersley_misc_interpolation}, generalized spiral points \cite{Brauchart_2014MCoM_qmc} and even random points fit quite well within our theoretical framework \cite{Sun_2025_monte}. The source code is available for reproducibility at the repository \cite{sun2026githubcode}.

We consider a sequence of STDs indexed by the parameter $t \in \{53, 75, 107, 155, 223\}$, which yields point sets $X_N$ with cardinalities $N$ ranging from 1,434 to 24,978. The associated mesh norm (fill distance) satisfies the scaling $h_{X_N} \sim N^{-1/2}$. To test the approximation accuracy, we compute the discrete $\ell_2(Y)$ error on a dense evaluation grid $Y$ ($|Y| = 52,978$). The error is approximated via the quadrature rule:
\begin{equation}
	\| \mathbf{e} \|_{\bld{L}_2(\mathbb{S}^2)} \approx \Big( \frac{4\pi}{|Y|} \sum_{\by \in Y} \| \mathbf{f}(\by) - \mathbf{s}(\by) \|_2^2 \Big)^{1/2}.
\end{equation}

Tables \ref{tab:Conv_Gauss} and \ref{tab:Conv_WE} summarize the convergence results for the smooth target Field 1 using restricted Gaussian and Wendland kernels, respectively. As established in Theorem \ref{thm:QMC_error}, setting the optimal shape parameter $\rho = \mathcal{O}(h^{1/2})$ for $\tau=0$ yields a convergence rate of $\mathcal{O}(h^{m/2})$.
We observe that the approximation error decays algebraically with respect to the mesh size, with rates improving significantly as the kernel order increases. Specifically, for kernels satisfying Assumption \ref{assump:kernel} with orders $m \in \{2, 4, 6, 8\}$, the observed convergence rates approach $\mathcal{O}(h^1)$, $\mathcal{O}(h^2)$, $\mathcal{O}(h^3)$, and $\mathcal{O}(h^4)$, respectively. These results demonstrate that the method attains the theoretical order of accuracy for smooth vector fields.

We further evaluate the method on Field 2 which possesses limited regularity. Figure \ref{fig.Err_Field2} displays the $\ell_2$ errors for both the combined field and its Helmholtz--Hodge components by using the proposed vector quasi-interpolation sampled at MD points with $\text{WE}_{3,2}$ kernel. In this scenario, Theorem \ref{thm:QMC_error} implies that the convergence is limited by the target regularity when $m \geq \sigma$. Consequently, we still employ the scaling parameter $\rho = \mathcal{O}(h^{1/2})$. As illustrated in Figure \ref{fig.Err_Field2}, the error decay saturates as the kernel order $m$ increases. The numerical results show a convergence rate of approximately $\mathcal{O}(h^3)$ for large $m$, which confirms that the error is dominated by the regularity of the target field rather than the kernel order. The numerical solutions and corresponding pointwise errors for Field 1 and Field 2 by using $\text{WE}_{3,2}$ are presented in Figure \ref{fig.NumerSolution}.

\begin{table}
	\centering
	\caption{Convergence results for the approximation of Field 1 using restricted Gaussian kernels of varying orders ($m= 2, 4, 6, 8$), where the shape parameter scales as $\rho = \mathcal{O}(h^{1/2})$ with $0.4h^{1/2},0.75h^{1/2},1h^{1/2},1.25h^{1/2}$.}
	
	\begin{tabular}{c *{8}{c}} 
		\toprule[1pt] 
		$N$ & $m=2$ & rate & $m=4$ & rate & $m=6$ & rate & $m=8$ & rate \\
		\midrule 
		
		1434  & 4.908e-02 &  & 5.965e-03 &  & 1.373e-03 &  & 2.527e-04 &  \\
		2852  & 3.511e-02 & 0.97 & 3.066e-03 & 1.94 & 5.113e-04 & 2.87 & 6.742e-05 & 3.84 \\
		5780  & 2.482e-02 & 0.98 & 1.537e-03 & 1.95 & 1.829e-04 & 2.91 & 1.706e-05 & 3.89 \\
		12092 & 1.724e-02 & 0.99 & 7.435e-04 & 1.97 & 6.182e-05 & 2.94 & 4.010e-06 & 3.92 \\
		24978 & 1.203e-02 & 0.99 & 3.628e-04 & 1.98 & 2.115e-05 & 2.96 & 9.575e-07 & 3.95 \\
		\bottomrule[1pt] 
	\end{tabular}
	\label{tab:Conv_Gauss}
\end{table}

\begin{table}
	\centering
	\caption{Convergence results for the  approximation of Field 1 using restricted WE$_{3,2}$ of varying orders ($m= 2, 4, 6,8$), where the shape parameter scales as $\rho = \mathcal{O}(h^{1/2})$ with $1.6h^{1/2},3.2h^{1/2}, 4.3h^{1/2}, 6.4h^{1/2}$.}
	\begin{tabular}{c *{8}{c}} 
		\toprule[1pt] 
		$N$ & $m=2$ & rate & $m=4$ & rate & $m=6$ & rate & $m=8$ & rate \\
		\midrule 
		1434  & 4.204e-02 &  & 5.329e-03 & & 1.172e-03 &  & 6.135e-04 &  \\
		2852  & 2.994e-02 & 0.99 & 2.700e-03 & 1.98 & 4.158e-04 & 3.01 & 1.640e-04 & 3.84 \\
		5780  & 2.113e-02 & 0.99 & 1.346e-03 & 1.97 & 1.451e-04 & 2.98 & 4.168e-05 & 3.88 \\
		12092 & 1.466e-02 & 0.99 & 6.492e-04 & 1.98 & 4.839e-05 & 2.98 & 9.861e-06 & 3.91 \\
		24978 & 1.023e-02 & 0.99 & 3.163e-04 & 1.98 & 1.639e-05 & 2.98 & 2.370e-06 & 3.93 \\
		\bottomrule[1pt] 
	\end{tabular}
	\label{tab:Conv_WE}
\end{table}

\begin{figure}
	\centering
	\begin{tikzpicture}
		\begin{groupplot}[
			group style={
				group size=3 by 1, 
				horizontal sep=0pt, 
				vertical sep=0pt,
			},
			width=4cm, height=5cm, 
			scale only axis,           
			xmode=log, ymode=log,
			xmin=0.003, xmax=0.2,
			ymin=1e-4, ymax=5e0,
			grid=both,
			major grid style={line width=0.2pt, draw=gray!40, dashed},
			minor grid style={line width=0.1pt, draw=gray!15, dotted},
			minor x tick num=3, minor y tick num=3,
			ticklabel style={font=\tiny},
			xlabel={$h$},
			xlabel style={font=\small, yshift=6pt},
			legend style={
				at={(0.98,0.05)}, 
				anchor=south east,
				font=\tiny, 
				cells={anchor=west},
				legend columns=1,
				inner sep=2pt,
				outer sep=1pt,
				draw=none, 
				fill opacity=0.8,
				text opacity=1
			},
			legend image post style={mark size=1.8pt}, 
			cycle list={
				{color={myred}, mark=triangle*, line width=0.8pt,mark size=2.5pt},
				{color={myblue}, mark=square*, line width=0.8pt,mark size=1.5pt},
				{color={mygreen}, mark=diamond*, line width=0.8pt,mark size=2.5pt}
			},
			]
			
			\nextgroupplot[
			ylabel={$\ell_2$ error}, 
			title={combined},
			title style={font=\large, yshift=-3pt},
			]
			
			\addplot table[x index=0, y index=2, col sep=space] {data_VecQI/field2comwe.txt};
			\addlegendentry{WE$_{3,2}$, $m=4$}
			
			\addplot table[x index=0, y index=3, col sep=space] {data_VecQI/field2comwe.txt};
			\addlegendentry{WE$_{3,2}$, $m=6$}
			
			\addplot table[x index=0, y index=4, col sep=space] {data_VecQI/field2comwe.txt};
			\addlegendentry{WE$_{3,2}$, $m=8$}

			\addplot[
			domain=0.01:0.06,
			samples=2,
			color=black,
			line width=0.5pt,  %
			dashed,
			forget plot,
			] {400 * x^3};
			
			\node[anchor=south west, font=\tiny] at (axis cs:0.035,0.01) {$h^{3}$};

			\nextgroupplot[
			title={div-free},
			title style={font=\large, yshift=-3pt},
			yticklabels={},                 
			ylabel={},
			axis y line*=right, 
			]

			
			\addplot table[x index=0, y index=2, col sep=space] {data_VecQI/field2divwe.txt};
			\addlegendentry{WE$_{3,2}$, $m=4$}
			
			\addplot table[x index=0, y index=3, col sep=space] {data_VecQI/field2divwe.txt};
			\addlegendentry{WE$_{3,2}$, $m=6$}
			
			\addplot table[x index=0, y index=4, col sep=space] {data_VecQI/field2divwe.txt};
			\addlegendentry{WE$_{3,2}$, $m=8$}

			\addplot[
			domain=0.01:0.06,
			samples=2,
			color=black,
			line width=0.5pt,  %
			dashed,
			forget plot,
			] {250 * x^3};
			
			\node[anchor=south west, font=\tiny] at (axis  cs:0.04,0.01) {$h^{3}$};

			\nextgroupplot[
			title={curl-free},
			title style={font=\large, yshift=-3pt},
			yticklabels={},                 
			ylabel={},
			axis y line*=right, 
			]
			
			\addplot table[x index=0, y index=2, col sep=space] {data_VecQI/field2curlwe.txt};
			\addlegendentry{WE$_{3,2}$, $m=4$}
			
			\addplot table[x index=0, y index=3, col sep=space] {data_VecQI/field2curlwe.txt};
			\addlegendentry{WE$_{3,2}$, $m=6$}
			
			\addplot table[x index=0, y index=4, col sep=space] {data_VecQI/field2curlwe.txt};
			\addlegendentry{WE$_{3,2}$, $m=8$}

			\addplot[
			domain=0.01:0.06,
			samples=2,
			color=black,
			line width=0.5pt, 
			dashed,
			forget plot,
			] {300 * x^3};
			
			\node[anchor=south west, font=\tiny] at (axis cs:0.04,0.01) {$h^{3}$};
			
			
			
		\end{groupplot}
	\end{tikzpicture}
	
	\captionsetup{font=normalsize}
	\caption{Numerical errors and convergence orders of vector quasi-interpolation for approximating Field 2 with different numbers of MD nodes, using divergence-free and curl-free kernels constructed by restricted $\text{WE}_{3,2}$ kernel satisfying assumption 2.1 ( $m=4,6,8$).}
	\label{fig.Err_Field2}
\end{figure}
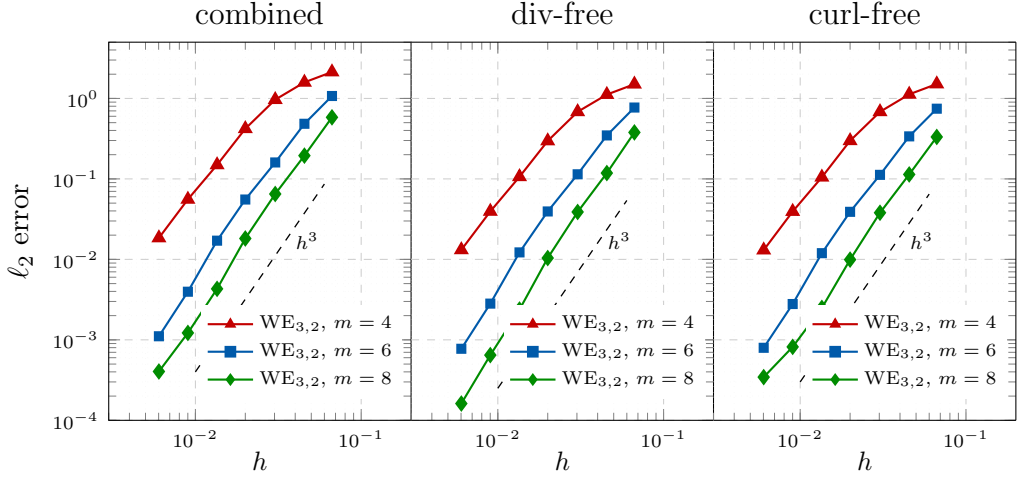

\begin{figure}
	\centering
	\begin{tikzpicture}
		\matrix (M) [
		matrix of nodes,
		nodes={inner sep=0pt, anchor=center}, 
		column sep=0.2cm, 
		row sep=0.5cm,   
		] {
			\includegraphics[width=0.3\textwidth, trim=2.3cm 0cm 2.3cm 0cm, clip]{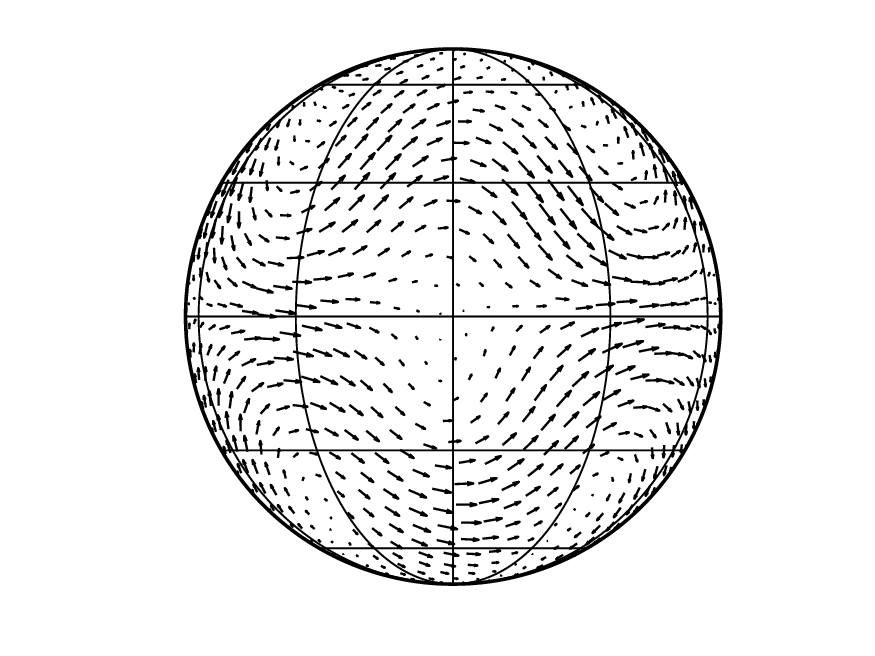} &
			\includegraphics[width=0.3\textwidth, trim=2.3cm 0cm 2.3cm 0cm, clip]{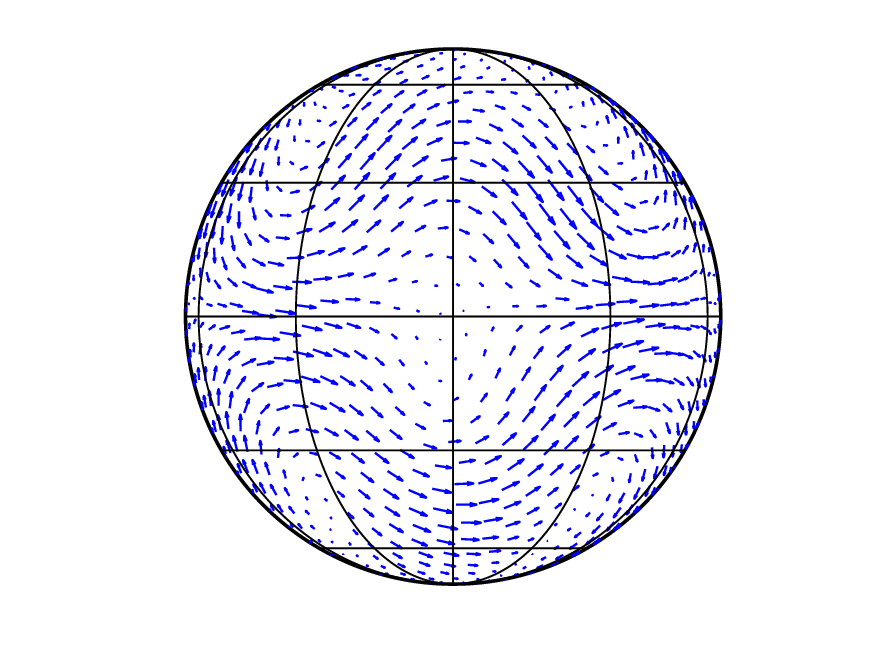} &
			\hspace*{-0.1cm}\includegraphics[width=0.36\textwidth, trim=1.2cm 0cm 1.2cm 0cm, clip]{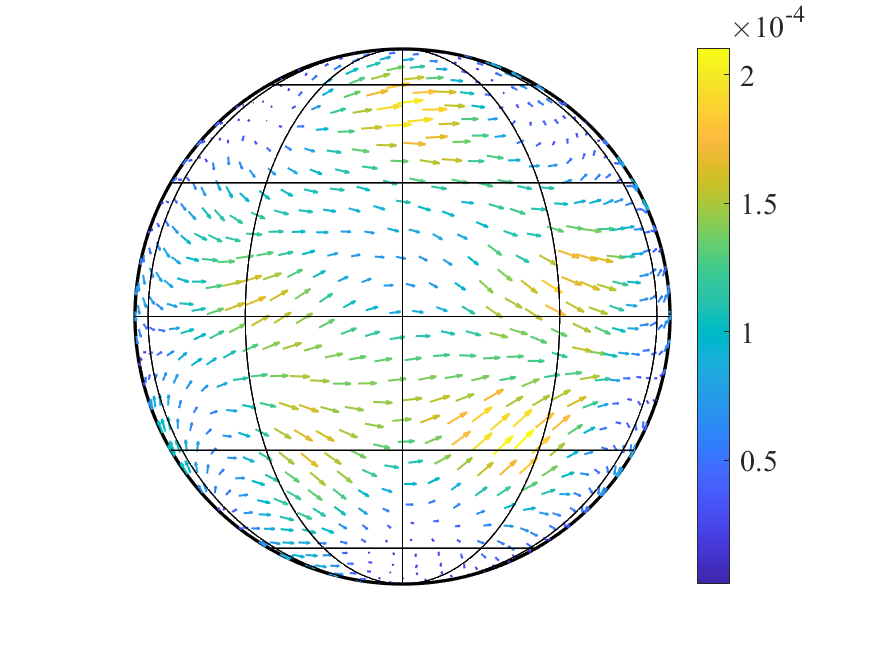} \\
			\includegraphics[width=0.3\textwidth, trim=2.3cm 0cm 2.3cm 0cm, clip]{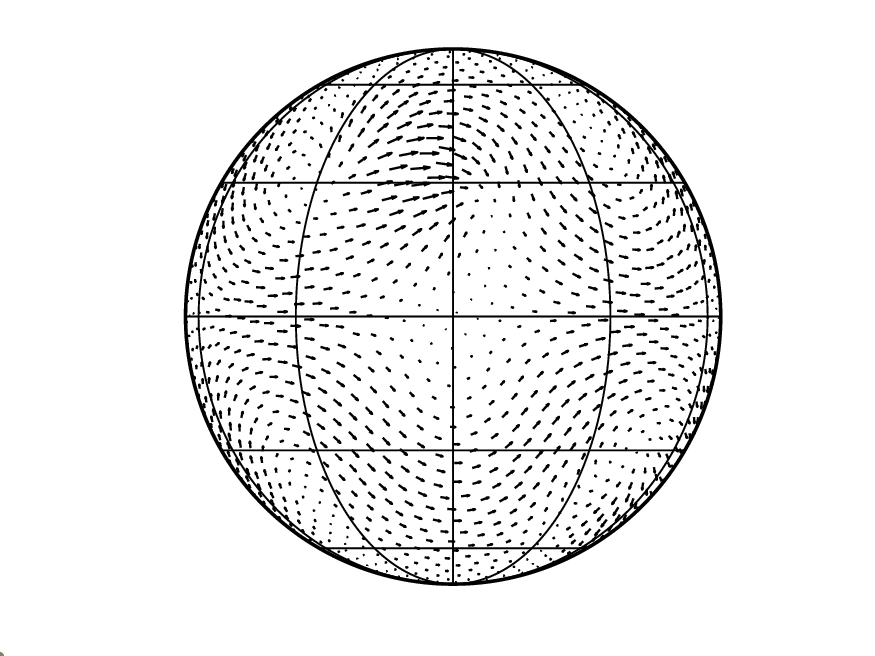} &
			\includegraphics[width=0.3\textwidth, trim=2.3cm 0cm 2.3cm 0cm, clip]{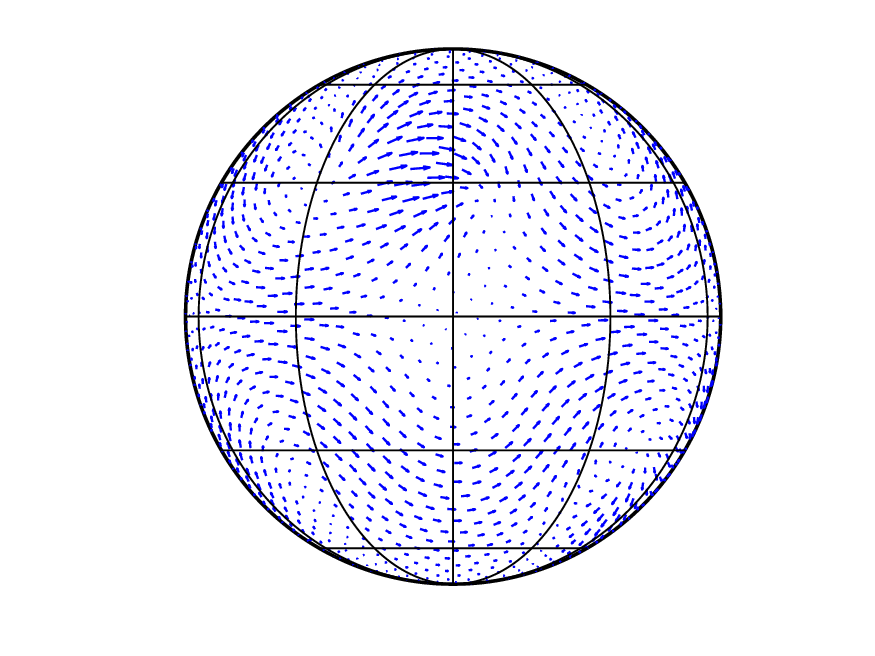} &
			\hspace*{-0.1cm}\includegraphics[width=0.36\textwidth, trim=1.2cm 0cm 1.2cm 0cm, clip]{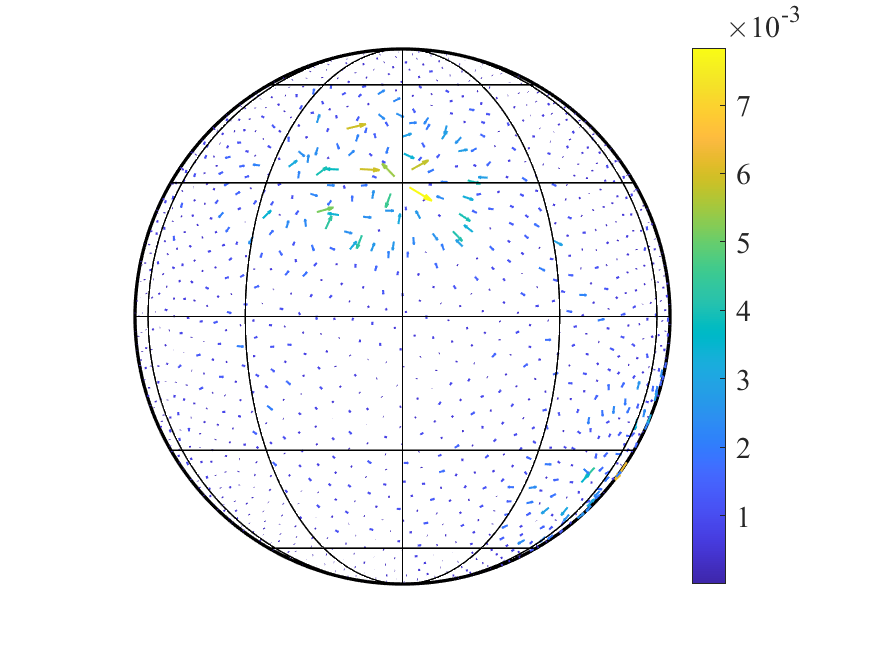} \\
		};

		\node [anchor=south, inner sep=1pt, yshift=-2mm, font=\small] at (M-1-1.north) {Exact};
		\node [anchor=south, inner sep=1pt, yshift=-2mm, font=\small] at (M-1-2.north) {Numerical};
		\node [anchor=south, inner sep=1pt, yshift=-2mm, font=\small,xshift=-0.3cm] at (M-1-3.north) {Error};
		
		\node [anchor=north, inner sep=1pt, yshift=2mm] at (M-1-2.south) {\textbf{(a) Field 1}};
		\node [anchor=north, inner sep=1pt, yshift=2mm] at (M-2-2.south) {\textbf{(b) Field 2}};	
	\end{tikzpicture}
	
	\captionsetup{font=normalsize}
	\caption{Comparison of the exact solution (left), numerical solution (middle), and pointwise error (right) for Field 1 and Field 2. Results were computed using the proposed vector quasi-interpolation scheme with $N=1434$ STD points for Field 1 and  $N=2,500$ MD nodes for Field 2. }
	\label{fig.NumerSolution}
\end{figure}

\subsection{Computational efficiency}
In this experiment, we benchmark the proposed spherical vector field quasi-interpolation (VQI) against the standard vector SBF interpolation developed by Fuselier and Wright \cite{Fuselier_2009SINUM_stability}. We use Field 2 as the target vector field. For the VQI method, we employ the restricted $\text{WE}_{3,2}$ kernel, which satisfies Assumption \ref{assump:kernel} with $m=8$. To ensure a fair comparison of approximation power, the SBF interpolation employs the same $\text{WE}_{3,2}$ kernel. Theoretical estimates suggest that for this limited smooth target field, both methods should asymptotically yield an error reduction rate of $\mathcal{O}(h^3)$, as convergence is limited by the regularity of the vector field.

We first explore the computational complexity of the two methods. As illustrated in Figure \ref{fig.CPU}(a), the proposed VQI method demonstrates linear scaling with respect to the number of nodes, $\mathcal{O}(N)$, a direct consequence of its explicit summation formulation. This contrasts with RBF interpolation, which requires the solution of a symmetric positive-definite linear system; this typically entails $\mathcal{O}(N^3)$ complexity for direct solvers or necessitates the use of preconditioned iterative methods.

Furthermore, we evaluate the computational cost required to achieve specific target error tolerances ($\text{Tol} \in \{3\cdot 10^{-4}, \dots, 2\cdot 10^{-3}\}$). Figure \ref{fig.CPU}(b) presents a work-precision diagram, plotting CPU time against the achieved accuracy. The results indicate that the VQI method is significantly more efficient; to attain a fixed error level, the VQI method requires approximately one-half of the computation time of the interpolation method. This efficiency gap widens as $N$ increases, which highlights the suitability of the proposed explicit approach for large-scale problems where solving dense linear systems becomes computationally prohibitive.

\begin{figure}[htbp]
	\centering
	\begin{tikzpicture}[scale=0.9]
		\begin{axis}[
			xmode=log,
			ymode=log,
			xmin=500, xmax=3.5e4,
			ymin=1, ymax=300,
			xlabel={$N$},
			xlabel style={
				at={(axis description cs:0.5,0)}, 
				yshift=-1.2em, 
				font=\footnotesize
			},
			ylabel={},  %
			ylabel style={
				rotate=-90, 
				at={(axis description cs:-0.12,0.5)}, 
				anchor=south,
				font=\footnotesize
			},
			grid=major,
			major grid style={
				line width=0.3pt, 
				draw=gray!60, 
				dash pattern=on 2pt off 2pt,
				opacity=0.6
			},
			minor grid style={
				line width=0.15pt, 
				draw=gray!40, 
				dash pattern=on 1pt off 1pt,
				opacity=0.4
			},
			xtick={1e3,1e4},
			ytick={1e0,1e1,1e2},
			minor x tick num=3,
			minor y tick num=3,
			tick label style={font=\footnotesize},
			legend style={
				at={(0.98,0.02)},
				anchor=south east,
				font=\small, 
				cells={anchor=west},
				legend columns=1,
				inner sep=2pt,
				outer sep=1pt,
				draw=black, 
				fill opacity=0.8,
				text opacity=1
			},
			legend image post style={mark size=1.8pt},
			title={CPU time of quasi-interpolation},
			title style={font=\large, yshift=-3pt},
			]

			\addplot[
			color={myblue}, 
			only marks,             
			mark =*,
			mark size= 2pt,
			] table[
			x index=0,
			y index=1,
			col sep=tab
			] {data_VecQI/CPU.txt};
			\addlegendentry{Quasi-interpolation};
			
			\addplot[
			color=myred,
			dashed,
			line width=1pt,            
			domain=800:3e4,            
			samples=100,               
			] {0.0015* x^1.13};
			\addlegendentry{$\mathcal{O}(N^{1.13})$};
		\end{axis}
	\end{tikzpicture}
	\hspace{0.4cm}
	%
	\begin{tikzpicture}[scale=0.9]
		\begin{axis}[
			xmode=log,
			ymode=log,
			x dir=reverse,
			xmin=1e-4, xmax=3e-3,
			ymin=1e1, ymax=600,
			xlabel={Target error ($\text{Tol}$)},
			xlabel style={
				at={(axis description cs:0.5,0)}, 
				yshift=-1.2em, 
				font=\footnotesize
			},
			ylabel={},
			ylabel style={
				rotate=0, 
				at={(axis description cs:-0.12,0.5)}, 
				anchor=south,
				font=\footnotesize
			},
			grid=major,
			major grid style={
				line width=0.3pt, 
				draw=gray!60, 
				dash pattern=on 2pt off 2pt,
				opacity=0.6
			},
			minor grid style={
				line width=0.15pt, 
				draw=gray!40, 
				dash pattern=on 1pt off 1pt,
				opacity=0.4
			},
			xtick={1e-3,1e-4},
			ytick={1e1,1e2},
			minor x tick num=3,
			minor y tick num=3,
			tick label style={font=\footnotesize},
			legend style={
				at={(0.98,0.02)}, 
				anchor=south east,
				font=\small, %
				cells={anchor=west},
				legend columns=1,
				inner sep=2pt,
				outer sep=1pt,
				draw=black, %
				fill opacity=0.8,
				text opacity=1
			},
			legend image post style={mark size=1.8pt},
			title={Comparison with SBF interpolation},
			title style={font=\large, yshift=-3pt},
			]

			\addplot[
			color={myblue}, 
			line width=1pt,              
			mark = *,
			mark size=1.5pt,
			] table[
			x index=0,
			y index=1,
			col sep=tab
			] {data_VecQI/computationtime.txt};
			\addlegendentry{Quasi-interpolation};
			
			\addplot[
			color={myred}, 
			line width=1pt,
			mark = triangle*,
			mark size=2pt
			] table[
			x index=0,
			y index=2,
			col sep=tab
			] {data_VecQI/computationtime.txt};
			\addlegendentry{Interpolation};
			
		\end{axis}
	\end{tikzpicture}
	\vspace{-0.3cm}
	\captionsetup{font=normalsize}
	\caption{ Evaluation of computational efficiency. \textbf{Left}: CPU time (s) versus the number of nodes $N$, confirming the linear complexity $\mathcal{O}(N)$ of the proposed quasi-interpolation. \textbf{Right}: CPU time required by quasi-interpolation (QI) versus SBF interpolation to achieve fixed target error tolerances.}
	\label{fig.CPU}
\end{figure}
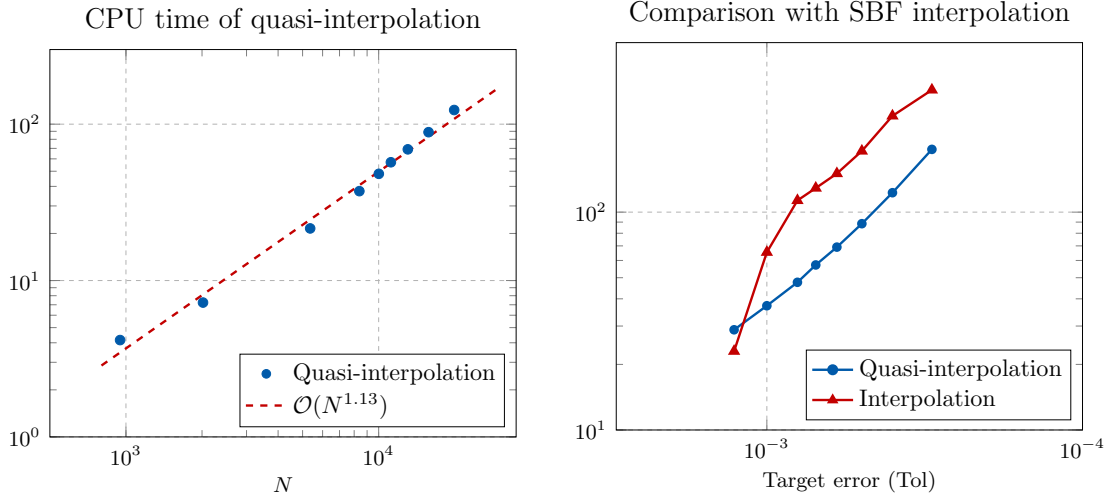

\subsection{Simulation of noisy data}

Finally, we evaluate the stability of the proposed method in the presence of noise. We introduce additive Gaussian noise to the vector field samples via
\begin{equation}
	\tilde{\mathbf{f}}(\mathbf{x}_i) = \mathbf{f}(\mathbf{x}_i) + \boldsymbol{\epsilon}_{i}, \quad \boldsymbol{\epsilon}_{i} \sim \mathcal{N}(0, \delta^2 \mathbf{I}),
\end{equation}
considering four distinct noise levels $\delta \in \{0.001, 0.01, 0.1, 0.5\}$. Approximation accuracy is quantified using the root mean square error (RMSE) computed on a dense evaluation grid of size $|Y|=52,978$. To ensure statistical reliability, all reported results represent the average of 30 independent realizations.

Figure \ref{fig.Noisydata} depicts the convergence profiles of both methods across varying noise intensities. A sharp contrast in stability is evident. The proposed VQI operator, acting as a discrete approximation to a convolution integral, inherently possesses smoothing properties that filter high-frequency noise components. Consequently, the VQI error decays monotonically with increasing node density $N$, even under significant noise ($\delta=0.5$). Notably, this robustness is achieved without the need for auxiliary regularization parameters or matrix ill-conditioning treatments, confirming the method's utility as an effective mollifier. Conversely, standard RBF interpolation enforces the strict interpolation condition, thereby incorporating the noise directly into the approximant. While the method converges for negligible noise, the error quickly saturates at higher noise levels, bounded below by the noise floor. Further error reduction in the interpolation setting would require abandoning exact interpolation in favor of smoothing splines, whereas the proposed VQI handles the noise naturally.

\begin{figure}
	\centering
	\begin{tikzpicture}[scale=0.9]
		\begin{axis}[
			xmode=log,
			ymode=log,
			xmin=50, xmax=1.8e4,
			ymin=1e-4, ymax=2e0,
			xlabel={$N$},
			xlabel style={
				at={(axis description cs:0.5,0)}, 
				yshift=-1.2em, 
				font=\footnotesize
			},
			ylabel={},
			ylabel style={
				rotate=-90, 
				at={(axis description cs:-0.12,0.5)}, 
				anchor=south,
				font=\footnotesize
			},
			grid=major,
			major grid style={
				line width=0.3pt, 
				draw=gray!60, 
				dash pattern=on 2pt off 2pt,
				opacity=0.6
			},
			minor grid style={
				line width=0.15pt, 
				draw=gray!40, 
				dash pattern=on 1pt off 1pt,
				opacity=0.4
			},
			xtick={1e2,1e3,1e4},
			ytick={1e-4,1e-3,1e-2,1e-1,1e0},
			minor x tick num=3,
			minor y tick num=3,
			tick label style={font=\footnotesize},
			legend style={
				at={(0.04,0.02)}, 
				anchor=south west,
				font=\small, %
				cells={anchor=west},
				legend columns=1,
				inner sep=2pt,
				outer sep=1pt,
				draw= none, %
				fill opacity=0.8,
				text opacity=1
			},
			legend image post style={mark size=1.8pt},
			title={Spherical quasi-interpolation},
			title style={font=\large, yshift=-3pt},
			]

			\addplot[
			color={myred}, 
			line width=0.5pt,              
			mark = triangle,
			mark size= 1.5pt,
			] table[
			x index=0,
			y index=1,
			col sep=tab
			] {data_VecQI/field2noise.txt};
			\addlegendentry{$\delta=0.5$};
			
			\addplot[
			color={myblue}, 
			line width=0.5pt,
			mark = square,
			mark size=0.75pt
			] table[
			x index=0,
			y index=2,
			col sep=tab
			] {data_VecQI/field2noise.txt};
			\addlegendentry{$\delta=0.1$};
			
			\addplot[
			color={mygreen}, 
			line width=0.5pt,
			mark = diamond,
			mark size=1.5pt
			] table[
			x index=0,
			y index=3,
			col sep=tab
			] {data_VecQI/field2noise.txt};
			\addlegendentry{$\delta=0.01$};
			
			\addplot[
			color={myorange}, 
			line width=0.5pt,
			mark = o,
			mark size=1pt
			] table[
			x index=0,
			y index=4,
			col sep=tab
			] {data_VecQI/field2noise.txt};
			\addlegendentry{$\delta=0.001$};
		\end{axis}
	\end{tikzpicture}
	\hspace{0.5cm}
	%
	\begin{tikzpicture}[scale = 0.9]
		\begin{axis}[
			xmode=log,
			ymode=log,
			xmin=50, xmax=1.8e4,
			ymin=1e-4, ymax=2e0,
			xlabel={$N$},
			xlabel style={
				at={(axis description cs:0.5,0)}, 
				yshift=-1.2em, 
				font=\footnotesize
			},
			ylabel={},  %
			ylabel style={
				rotate=-90, 
				at={(axis description cs:-0.12,0.5)}, 
				anchor=south,
				font=\footnotesize
			},
			grid=major,
			major grid style={
				line width=0.3pt, 
				draw=gray!60, 
				dash pattern=on 2pt off 2pt,
				opacity=0.6
			},
			minor grid style={
				line width=0.15pt, 
				draw=gray!40, 
				dash pattern=on 1pt off 1pt,
				opacity=0.4
			},
			xtick={1e2,1e3,1e4},
			ytick={1e-4,1e-3,1e-2,1e-1,1e0},
			minor x tick num=3,
			minor y tick num=3,
			tick label style={font=\footnotesize},
			legend style={
				at={(0.04,0.02)}, 
				anchor=south west,
				font=\small, 
				cells={anchor=west},
				legend columns=1,
				inner sep=2pt,
				outer sep=1pt,
				draw = none, 
				fill opacity=0.8,
				text opacity=1
			},
			legend image post style={mark size=1.8pt},
			title={SBF interpolation},
			title style={font=\large, yshift=-3pt},
			]

			\addplot[
			color={myred}, 
			line width=0.5pt,              
			mark = triangle,
			mark size= 1.5pt,
			] table[
			x index=0,
			y index=1,
			col sep=tab
			] {data_VecQI/field2noise1.txt};
			\addlegendentry{$\delta=0.5$};
			
			\addplot[
			color={myblue}, 
			line width=0.5pt,
			mark = square,
			mark size=0.75pt
			] table[
			x index=0,
			y index=2,
			col sep=tab
			] {data_VecQI/field2noise1.txt};
			\addlegendentry{$\delta=0.1$};
			
			\addplot[
			color={mygreen}, 
			line width=0.5pt,
			mark = diamond,
			mark size=1.5pt
			] table[
			x index=0,
			y index=3,
			col sep=tab
			] {data_VecQI/field2noise1.txt};
			\addlegendentry{$\delta=0.01$};
			
			\addplot[
			color={myorange}, 
			line width=0.5pt,
			mark = o,
			mark size=1pt
			] table[
			x index=0,
			y index=4,
			col sep=tab
			] {data_VecQI/field2noise1.txt};
			\addlegendentry{$\delta=0.001$};
			
		\end{axis}
	\end{tikzpicture}
	\vspace{-0.3cm}
	\captionsetup{font=normalsize}
	\caption{RMSE for the proposed vector quasi-interpolation and standard SBF interpolation as a function of the number of STD nodes $N$. Results are shown for four distinct Gaussian noise levels: $\delta \in \{0.001, 0.01, 0.1, 0.5\}$.}
	\label{fig.Noisydata}
\end{figure}
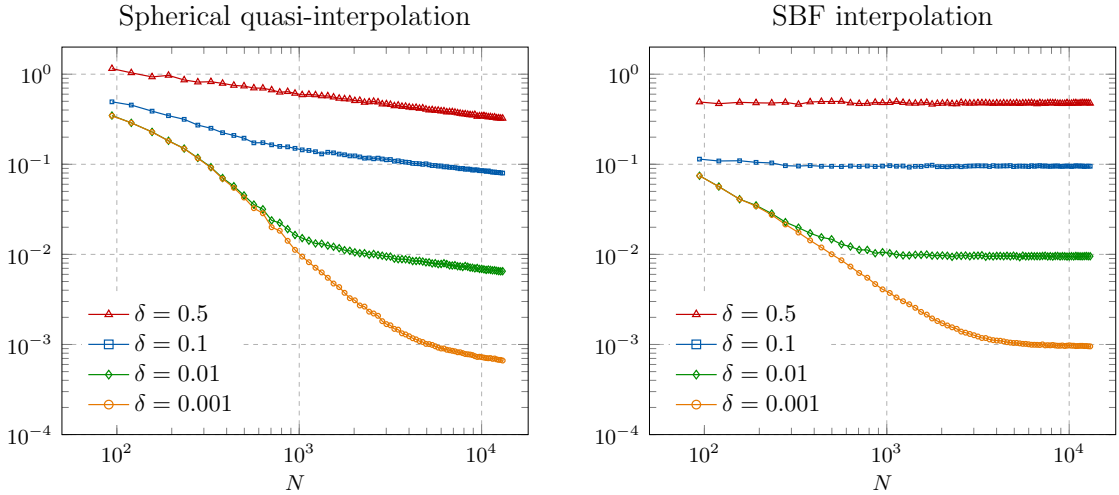


\bibliographystyle{plain}
\bibliography{reference}

\end{document}